\documentclass[a4paper,10pt,psamsfonts,draft]{amsart}

\usepackage{todonotes}
\presetkeys{todonotes}{fancyline, color=blue!30}{}
\usepackage{enumitem}

\usepackage[utf8]{inputenc} 

\usepackage{amsmath}
\usepackage{amstext}
\usepackage{amsfonts}
\usepackage{amsthm}
\usepackage{amssymb}

\usepackage{dsfont}            
\usepackage{mathrsfs}          
\usepackage{tensor}            

\usepackage{enumerate}

\usepackage{hyperref}
\hypersetup{colorlinks}
\usepackage[toc,page]{appendix}
\usepackage{comment}

\usepackage{color}

\newcommand{\R}{{\mathds{R}}}

\newcommand{\E}{{\mathds{E}}}
\newcommand{\N}{{\mathds{N}}}

\renewcommand{\P}{{\mathds{P}}} 

\newcommand{\LB}{{\mathcal{L}}}

\newcommand{\diffout}[1]{\,\mathrm{d}#1}
\newcommand{\diffin}[1]{\mathrm{d}#1}

\newcommand{\id}{\mathrm{id}}

\newcommand{\bD}{{\mathds D}}
\newcommand{\bE}{{\mathds E}}

\newcommand{\bM}{{\mathds M}}
\newcommand{\bN}{{\mathds N}}

\newcommand{\bP}{{\mathds P}}

\newcommand{\bR}{{\mathds R}}

\newcommand{\cB}{\ensuremath{\mathcal B}}
\newcommand{\cC}{\ensuremath{\mathcal C}}

\newcommand{\cF}{\ensuremath{\mathcal F}}

\newcommand{\cL}{\ensuremath{\mathcal L}}

\newcommand{\cN}{\ensuremath{\mathcal N}}
\newcommand{\cO}{\ensuremath{\mathcal O}}
\newcommand{\cP}{\ensuremath{\mathcal P}}

\newcommand{\cS}{\ensuremath{\mathcal S}}

\newcommand{\pred}{\ensuremath{\operatorname{pr}}}
\newcommand{\dl}{\mathrm{d}}
\newcommand{\triple}{{\vert\kern-0.25ex\vert\kern-0.25ex\vert}}

\newcommand{\Lip}{\operatorname{Lip}}

\newcommand{\NN}{\mathbf N}

\renewcommand{\geq}{\ensuremath{\geqslant}}
\renewcommand{\leq}{\ensuremath{\leqslant}}

\theoremstyle{plain}

\newtheorem{definition}{Definition}[section]
\newtheorem{theorem}[definition]{Theorem}
\newtheorem{lemma}[definition]{Lemma}
\newtheorem{corollary}[definition]{Corollary}
\newtheorem{prop}[definition]{Proposition}
\newtheorem{assumption}[definition]{Assumption}

\newtheorem{notation}[definition]{Notation}

\theoremstyle{definition}
\newtheorem{remark}[definition]{Remark}

\newtheorem{example}[definition]{Example}

\begin{document}

\title[Malliavin regularity and weak approximation]
{Malliavin regularity and weak approximation of semilinear SPDE with Lévy noise}

\author[A.~Andersson]{Adam Andersson}
\address{
Adam Andersson\\
Syntronic Software Innovations, Lindholmspiren 3B, SE-417 56 Gothenburg, Sweden 
}
\email{adan@syntronic.com }

\author[F.~Lindner]{Felix Lindner}
\address{Felix Lindner\\
University of Kassel\\
Institute of Mathematics\\
Heinrich-Plett-Str.~40,
34132~Kassel, Germany}
\email{lindner@mathematik.uni-kassel.de}

\begin{abstract}
We investigate the weak order of convergence for space-time discrete approximations of semilinear parabolic stochastic evolution equations driven by additive square-integrable Lévy noise.
To this end, the Malliavin regularity of the solution is analyzed and recent results on refined Malliavin-Sobolev spaces from the Gaussian setting are extended to a Poissonian setting.
For a class of path-dependent test functions, we obtain that the weak rate of convergence is twice the strong rate.
\end{abstract}


\maketitle


\section{Introduction}

Stochastic partial differential equations (SPDE) with Lévy noise occur in various applications, 
ranging from environmental pollution models \cite{kallianpur1995} to the statistical theory of turbulence \cite{birnir2013}, to mention only two examples.
In the context of the numerical approximation of the solution processes of such equations, the quantity of interest is typically the expected value of some functional of the solution
and one is thus interested in the weak convergence rate of the considered numerical scheme.
While the weak convergence analysis for numerical approximations of SPDE with Gaussian noise is meanwhile relatively far developed, see, e.g., \cite{AnderssonKovacsLarsson,
AnderssonKruseLarsson,
AnderssonLarsson2016,
Brehier2017,
BrehierDebussche2017,
BrehierGoudenege2018,
BrehierHairerStuart, 
Conus2014,
debussche2011,
Hausenblas2010,
Hefter2016,
JentzenNauroisWelti2015,
JentzenNauroisWelti2017,
Jentzen2015,
LangPettersson,
WangGan}, 
available results for non-Gaussian Lévy noise have been restricted to linear equations so far \cite{AnderssonLindner2017a,Barth2016,KovLinSch2015,LindnerSchilling}.
In this article, we analyze for the first time the weak convergence rate of numerical approximations for a class of semi-linear SPDE with non-Gaussian Lévy noise.

We consider equations of the type
\begin{align}\label{eq:SPDE_additive}
  \diffout X(t)+AX(t)\diffout t
  =
  F(X(t))\diffout t
  +
  \diffout L(t),
  \quad
  t\in[0,T],
  \quad
  X(0)=X_0,
\end{align}
where $X$ takes values in a separable real Hilbert space $H$ and $A\colon D(A)\subset H\to H$ is an unbounded linear operator such that $-A$ generates an analytic semigroup $(S(t))_{t\geq0}\subset\cL(H)$.
By $\dot H^\rho$, $\rho\in\bR$, we denote the smoothness spaces associated to $A$ via $\dot H^{\rho}=D(A^{\frac\rho2})$, see Subsection~\ref{subsec:operators} for details.
The driving Lévy process $L=(L(t))_{t\in[0,T]}$ is assumed to be $\dot H^{\beta-1}$-valued for some regularity parameter $\beta\in (0,1]$, square-integrable with mean zero, and of pure jump type.
The nonlinearity $F\colon H\to\dot H^{\beta-1}$ is supposed to satisfy suitable Lipschitz conditions. The precise assumptions are stated in Subsection~\ref{subsec:Malliavin_prel} and \ref{subsec:assumptionSPDE}.
We remark that for a strong convergence analysis one could allow $F$ to be only $\dot H^{\beta-2}$-valued, but to obtain a weak convergence rate which is twice the strong rate we need to assume more than that.
Our main example for the abstract equation \eqref{eq:SPDE_additive} is the semilinear heat equation
\begin{equation}\label{eq:SPDE_example}
\left\{
\begin{aligned}
&\dot u(t,\xi)-\Delta_\xi\;\! u(t,\xi)= f(u(t,\xi))+\dot \eta(t,\xi),\quad &&(t,\xi)\in[0,T]\times \mathcal O,\\
&u(t,\xi)=0, &&(t,\xi)\in[0,T]\times\partial\mathcal O,\\
&u(0,\xi)=u_0(\xi), &&\xi\in\mathcal O.
\end{aligned}
\right.
\end{equation}
Here $\mathcal O\subset \mathds R^{d}$ is an open, bounded, convex, polygonal/polyhedral domain, $d\in\{1,2,3\}$, $f\colon\bR\to\bR$ is twice continuously differentiable with bounded derivatives, and $\dot\eta$ is an impulsive space-time noise, cf.~Example~\ref{ex:A}.
The discretization in space is performed by a standard finite element method and in time by an implicit Euler method, cf.~Subsection~\ref{subsec:mainresult}.

Several approaches to analyzing the weak error of numerical approximations of SPDE can be found in the literature.
We follow the the approach from 
\cite{AnderssonKovacsLarsson,AnderssonKruseLarsson,AnderssonLindner2017a,Barth2016,kruse2013}, 
which is based on duality principles in Malliavin calculus.
We remark that Malliavin calculus for Poisson or Lévy noise is fundamentally different from that for Gaussian noise.
Our analysis heavily relies on the results on Hilbert space-valued Poisson Malliavin calculus from \cite{AnderssonLindner2017a}.
Following the ideas in \cite{Last2014,Picard1996a}, the Malliavin derivative in \cite{AnderssonLindner2017a} is in fact a finite difference operator
\begin{align}\label{eq:Dintro}
D\colon L^0(\Omega;H)\to L^0(\Omega\times[0,T]\times U;H),
\end{align}
where $(\Omega,\cF,\bP)$ is the underlying probability space 
and $U=\dot H^{\beta-1}$ is the state space of the Lévy process $L$, endowed with the Borel-$\sigma$-algebra $\cB(U)$ and the Lévy measure $\nu$ of $L$.
Starting with the operator \eqref{eq:Dintro}, one can in a second step define Malliavin-Sobolev-type spaces as classes of $H$-valued random variables satisfying certain integrability properties together with their Malliavin derivatives, cf.~Subsections~\ref{subsec:Malliavin_prel} and \ref{subsec:regX}.

In this article, we extend the strategy for semilinear SPDE from \cite{AnderssonKovacsLarsson,AnderssonKruseLarsson} to Poisson noise and analyze the weak approximation error in a framework of Gelfand triples of refined Malliavin-Sobolev spaces $\bM^{1,p,q}(H)\subset L^2(\Omega;H)\subset(\bM^{1,p,q}(H))^*$, see Subsection~\ref{subsec:regX} for the definition of these spaces. 
We first investigate in Section~\ref{sec:reg} the Malliavin regularity of the mild solution $X=(X(t))_{t\geq0}$ to Eq.~\eqref{eq:SPDE_additive}.
We start by proving in Proposition~\ref{prop:DX} that the Malliavin derivative $DX(t)$ of $X(t)$ satisfies for all $t\in[0,T]$ the equality
\begin{equation}\label{eq:DX2}
\begin{aligned}
D_{s,x}X(t)
&=
\mathds{1}_{s\leq t}\cdot\int_s^tS(t-r)\big[F\big(X(r)+D_{s,x}X(r)\big)-F\big(X(r)\big)\big]\,\dl r\\
&\quad +\mathds{1}_{s\leq t}\cdot S(t-s)x
\end{aligned}
\end{equation}
$\bP\otimes\dl s\otimes\nu(\dl x)$-almost everywhere on $\Omega\times[0,T]\times U$. The terms on the right hand side are understood to be zero for $s>t$. Based on 
this equality 
we derive in Proposition~\ref{thm:reg2} and \ref{thm:reg3} suitable integrability and time regularity properties of $DX(t)$ by using  Gronwall-type arguments.
The regularity results from Section~\ref{sec:reg} are then used in  Section~\ref{sec:weak} for the analysis of the weak  error $\bE [f(\tilde X_{h,k})-f(X)]$, where $X_{h,k}=(\tilde X_{h,k}(t))_{t\in[0,T]}$, $h,k\in(0,1)$, are time interpolated numerical approximations of $X$. We use a standard finite element method with maximal mesh size $h$ for the discretization in space and an implicit Euler method with step size $k$ for the discretization in time. 
For finite Borel measures $\mu_1,\ldots,\mu_n$ on $[0,T]$, we consider path-dependent functionals $f\colon L^1([0,T],\sum_{i=1}^n\mu_i;H)\to\bR$ of the form 
$f(x)=\varphi\big(\int_{[0,T]}x(t)\,\mu_1(\dl t),\ldots,\int_{[0,T]}x(t)\,\mu_n(\dl t)\big)$, where
$\varphi\colon\bigoplus_{i=1}^n\!H\to\bR$ is assumed to be Fréchet differentiable with globally Lipschitz continuous derivative mapping $\varphi'\colon \bigoplus_{i=1}^n\!H\to\cL\big(\bigoplus_{i=1}^n\!H,\bR\big)$.
Our main result, Theorem~\ref{thm:weak2}, states that for all $\gamma\in[0,\beta)$ there exists a finite constant $C$ such that
\begin{align}\label{eq:weak_conv}
|\bE[f(\tilde X_{h,k})-f(X)]|\leq C\,(h^{2\gamma}+k^\gamma),\quad h,k\in(0,1).
\end{align}
For the considered class of test functions, the weak rate of convergence is thus twice the strong rate. 
The idea of the proof 
is to exploit the Malliavin regularity of $X$ and $\tilde X_{h,k}$ in order to estimate the weak error $|\bE[f(\tilde X_{h,k})-f(X)]|$ in terms of the norm of the error $\tilde X_{h,k}(t)-X(t)$ in the dual space $(\bM^{1,p,q}(H))^*$, for suitable exponents $p,q\in[2,\infty)$. 
As an exemplary application, we consider in Corollary~\ref{cor:cov}  the approximation of covariances $\mathrm{Cov}(\langle X(t_1),\psi_1\rangle,\langle X(t_2),\psi_2\rangle)$, $t_1,t_2\in[0,T]$, $\psi_1,\psi_1\in H$ of the solution process.

We remark that weak error estimates for SPDE involving path-dependent functionals have been derived so far only in \cite{AnderssonKovacsLarsson,AnderssonLindner2017a,BrehierHairerStuart}.
Our setting allows for integral-type functionals as well as for functionals of the form $f(x)=\varphi(x(t_1),\ldots,x(t_n))$, where $x=(x(t))_{t\in[0,T]}$ is  an $H$-valued path, $0\leq t_1\leq\ldots\leq t_n\leq T$, and $\varphi\colon\bigoplus_{j=1}^n\!H\to\bR$. 



The paper is organized as follows: 
In Section~\ref{sec:prel} we collect some general notation (Subsection~\ref{subsec:operators}), introduce the precise assumptions on the Lévy process $L$ (Subsection~\ref{subsec:square}), and review fundamental concepts and results from Hilbert space-valued Poisson Malliavin calculus (Subsection~\ref{subsec:Malliavin_prel}).
Section~\ref{sec:reg} is concerned with the Malliavin regularity of the mild solution $X$ to Eq.~\eqref{eq:SPDE_additive}. 
Here we first describe in detail our assumptions on the considered equation (Subsection~\ref{subsec:assumptionSPDE}) before we analyse the regularity of $X$ (Subsection~\ref{subsec:regX}) and derive some auxiliary results concerning refined Malliavin-Sobolev spaces (Subsection~\ref{subsec:auxiliary}). The weak convergence analysis is found in Section~\ref{sec:weak}, where we present the numerical scheme and our main result (Subsection~\ref{subsec:mainresult}), analyze the regularity of the approximation process (Subsection~\ref{subsec:weak2}) as well as convergence in negative order Malliavin-Sobolev spaces (Subsection~\ref{subsec:weak3}), and finally prove the main result by combining the results previously collected (Subsection~\ref{subsec:weak4}.)

\section{Preliminaries}
\label{sec:prel}

\subsection{General notation}
\label{subsec:operators}
If $(U,\|\cdot\|_U,\langle\cdot,\cdot\rangle_U)$ and $(V,\|\cdot\|_V,\langle\cdot,\cdot\rangle_V)$ are separable real Hilbert spaces, we denote by $\LB(U,V)$ and $\LB_2(U,V)\subset \LB(U,V)$ the spaces of bounded linear operators and Hilbert-Schmidt operators from $U$ to $V$, respectively. 
By $\cC^1(U,V)$ we denote the space of Fréchet differentiable functions $f\colon U\to V$ with continuous derivative $f'\colon U\to\cL(U,V)$. 
In the special case $V=\bR$ we identify $\cL(U,\bR)$ with $U$ via the Riesz isomorphism and consider $f'$ as a $U$-valued mapping. 
The Lipschitz spaces 
\begin{align*}
\Lip^0(U,V)&:=\{f\in \cC(U,V):|f|_{\Lip^0(U,V)}<\infty\},\\
\Lip^1(U,V)&:=\{f\in \cC^1(U,V):|f|_{\Lip^0(U,V)}+|f|_{\text{Lip}^1(U,V)}<\infty\},
\end{align*}
are defined in terms of the semi-norms
\begin{align*}
|f|_{\Lip^0(U,V)}&:=
\sup\Big(\Big\{\tfrac{\|f(x)-f(y)\|_V}{\|x-y\|_U}:x,y\in U,\,x\neq y\Big\}\cup\{0\}\Big),\\
|f|_{\Lip^1(U,V)}&:=
\sup\Big(\Big\{\tfrac{\|f'(x)-f'(y)\|_{\cL(U,V)}}{\|x-y\|_U}:x,y\in U,\,x\neq y\Big\}\cup\{0\}\Big),
\end{align*}
compare, e.g., \cite[Sec.~1.2]{Conus2014}. We also use the norm  $\|f\|_{\Lip^0(U,V)}:=\|f(0)\|_V+|f|_{\Lip^0(U,V)}$.
If $(S,\cS,m)$ is a $\sigma$-finite measure space and $(X,\|\cdot\|_X)$ is a Banach space,
we denote by $L^0(S;X):=L^0(S,\cS,m;X)$ the space of (equivalence classes of) strongly $\cS$-measurable functions $f\colon S\to X$. As usual, we identify functions which coincide $m$-almost everywhere. The space $L^0(S;X)$ is endowed with the topology of local convergence in measure.
For $p\in[1,\infty]$, we denote by $L^p(S;X):=L^p(S,\cS,m;X)$ the subspace of $L^0(S;X)$ consisting of all (equivalence classes of) strongly $\cS$-measurable mappings $f\colon S\to X$ such that $\|f\|_{L^p(S;X)}:=\big(\int_S\|f(s)\|_X^p\,m(\dl s)\big)^{1/p}<\infty$ if $p\in[1,\infty)$ and  $\|f\|_{L^\infty(S;X)}:=\operatorname{ess\,sup}_{s\in S}\|f(s)\|_X<\infty$ if $p=\infty$. By $\lambda$ we denote one-dimensional Lebesgue measure and we sometimes also write $\lambda(\dl t)$, $\dl t$, $\lambda(\dl s)$, $\dl s$ etc.\ in place of $\lambda$ to improve readability.
\color{blue}
\color{black}

 

\subsection{L\'evy processes and Poisson random measures}
\label{subsec:square}

Here we describe in detail the setting concerning the driving process $L$ in Eq.~\eqref{eq:SPDE_additive}.
Our standard reference for Hilbert space-valued Lévy processes is \cite{PesZab2007}.

\begin{assumption}\label{setting}
The following setting is considered throughout the article.
\begin{itemize}[leftmargin=7mm]
\item 
$(\Omega,\cF,\P)$ is a complete probability space. The $\sigma$-algebra $\cF$ coincides with the $\P$-completion of the $\sigma$-algebra $\sigma(L(t):t\in[0,T])$ generated by the Lévy process $L$ introduced below.
\item
$L=(L(t))_{t\in[0,T]}$ is a Lévy process defined on $(\Omega,\cF,\P)$, 
taking values in a separable real Hilbert space $(U,\|\cdot\|_U,\langle\cdot,\cdot\rangle_U)$.
Here $T\in(0,\infty)$ is fixed.
We assume that $L$ is square-integrable with mean zero, i.e., $L(t)\in L^2(\Omega;U)$ and $\bE( L(t))=0$, and that the Gaussian part of $L$ is zero. 
\item
$(H,\|\cdot\|,\langle\cdot,\cdot\rangle)$ is a further separable real Hilbert space.
\end{itemize}
\end{assumption}

The jump intensity measure (Lévy measure) $\nu\colon\cB(U)\to[0,\infty]$ of a general $U$-valued Lévy process $L$ satisfies $\nu(\{0\})=0$ and $\int_U\min(\|x\|_U^2,1)\,\nu(\dl x)<\infty$, cf.~\cite[Section 4]{PesZab2007}.
Due to our square integrability assumption on $L$ we additionally have
\begin{align}\label{eq:ass_nu}
|\nu|_2 := \Big(\int_{U}\|y\|_U^2\,\nu(\dl y)\Big)^{\frac12}<\infty,
\end{align}
see, e.g., \cite[Theorem~4,47]{PesZab2007}. As a further consequence of our assumptions on $L$, the characteristic function of $L(t)$ is of given by
\begin{align}\label{eq:L_char_fn}
\bE e^{i\langle x,L(t)\rangle_U}=\exp\Big(-t\int_U\big(1-e^{i\langle x,y\rangle_U}+i\langle x,y\rangle_U\big)\,\nu(\dl y)\Big),\quad x\in U,
\end{align}
cf.~\cite[Theorem 4.27]{PesZab2007}. Conversely, every $U$-valued Lévy process $L$ satisfying \eqref{eq:ass_nu} and \eqref{eq:L_char_fn} is square-integrable with mean zero and vanishing Gaussian part.

We always consider a fixed càdlàg (right continuous with left limits) modification of $L$. 
The jumps of $L$ determine a Poisson random measure on $\cB([0,T]\times U)$ as follows: For $(\omega,t)\in\Omega\times(0,T]$ we denote by $\Delta L(t)(\omega):=L(t)(\omega)-\lim_{s\nearrow t}L(s)(\omega)\in U$ the jump of a trajectory of $L$ at time $t$. Then
\begin{equation}\label{eq:LjumpPRM}
N(\omega):=\sum_{t\in(0,T]:\Delta L(t)(\omega)\neq 0}\delta_{(t,\Delta L(t)(\omega))},\quad\omega\in\Omega,
\end{equation}
defines a Poisson random measure $N$ on $\cB([0,T]\times U)$ with intensity measure $\lambda\otimes\nu$, where $\delta_{(t,x)}$ denotes Dirac measure at $(t,x)\in[0,T]\times U$ and $\nu$ is the Lévy measure of $L$. This follows, e.g., from Theorem~6.5 in \cite{PesZab2007} together with Theorems~4.9, 4.15, 4.23 and Lemma 4.25 therein. 
It the context of Poisson Malliavin calculus it is useful to consider $N$ as a random variable with values in the space $\NN=\NN([0,T]\times U)$ of all $\sigma$-finite $\bN_0\cup\{+\infty\}$-valued measures on $\cB([0,T]\times U)$. It is endowed with the  $\sigma$-algebra $\mathcal{N}=\mathcal N([0,T]\times U)$ generated by the mappings $\NN\ni\mu\mapsto\mu(B)\in\bN_0\cup\{+\infty\}$, $B\in\cB([0,T]\times U)$.

We now list some important notation used in the present context.

\begin{notation}\label{notation}
The following notation is used throughout the article.
\begin{itemize}[leftmargin=7mm]
\item 
$\nu\colon\cB(U)\to[0,\infty]$
and $(U_0,\|\cdot\|_{U_0},\langle\cdot,\cdot\rangle_{U_0})$ are the Lévy measure and the reproducing kernel Hilbert space of $L$, respectively; cf.~\cite[Definition 4.28 and 7.2]{PesZab2007}.
\item 
$N\colon\Omega\to\NN$ is the Poisson random measure (Poisson point process) on $[0,T]\times U$ determined by the jumps of $L$ as specified in Eq.~\eqref{eq:LjumpPRM} above.
The compensated Poisson random measure is denoted by $\tilde N:=N-\lambda\otimes\nu$, i.e., $\tilde N(B)=N(B)-(\lambda\otimes\nu)(B)$ for all $B\in\cB([0,T]\times U)$ with $(\lambda\otimes\nu)(B)<\infty$
\item
$(\cF_t)_{t\in[0,T]}$ is the filtration given by 
$\cF_t:=\bigcap_{u\in(t,T]}\tilde\cF_u,$
where $\tilde\cF_u$ is the $\bP$-completion of $\sigma(L(s):s\in[0,u])$. 
\item
For $p\in\{0\}\cup[1,\infty]$ set $L^p(\Omega;H):=L^p(\Omega,\cF,\bP;H)$ and $L^p(\Omega\times[0,T]\times U;H):=$ $L^p(\Omega\times[0,T]\times U,\cF\otimes\cB([0,T]\times U),\bP\otimes\lambda\otimes\nu;H)$. Moreover,
$\cP_T\subset\cF\otimes\cB([0,T])$  denotes the $\sigma$-algebra of predictable sets w.r.t.\ to $(\cF_t)_{t\in[0,T]}$ and we further set 
$
L^2_{\pred}(\Omega\times[0,T]\times U; H):=L^2\big(\Omega\times[0,T]\times U,\cP_T\otimes\cB(U),\P\otimes\lambda\otimes\nu;H\big).
$
\end{itemize}
\end{notation}


We end this section by recalling some basics on stochastic integration w.r.t.\ $L$ and $\tilde N$, cf.~\cite{PesZab2007}.
The $H$-valued $L^2$ stochastic integral $\int_0^T\Phi(s)\,\dl L(s)$ w.r.t.\ $L$ is defined for all $\Phi\in L^2_{\pred}(\Omega\times[0,T];\cL_2(U_0,H)):=L^2(\Omega\times[0,T],\cP_T,\bP\otimes\lambda;\cL_2(U_0,H))$, and we have the It\^o isometry
$\bE\big\|\int_0^T\Phi(s)\,\dl L(s)\big\|^2=\int_0^T\bE\|\Phi(s)\|_{\cL_2(U_0,H)}^2\dl s$.
The $H$-valued $L^2$ stochastic integral $\int_0^T\int_U\Phi(s,x)\,\tilde N(\dl s,\dl x)$ w.r.t.\ $\tilde N$ is defined for all $\Phi\in L^2_{\pred}(\Omega\times[0,T]\times U;H)$, and here it holds that
$\bE\big\|\int_0^T\int_U\Phi(s,x)\,\tilde N(\dl s,\dl x)\big\|^2=\int_0^T\int_U\bE\|\Phi(s,x)\|^2\nu(\dl x)\dl s$.
As usual, we set $\int_0^t\Phi(s)\,\dl L(s):=\int_0^T\mathds{1}_{(0,t]}(s)\Phi(s)\,\dl L(s)$ and $\int_0^t\int_U\Phi(s,x)\,\tilde N(\dl s,\dl x):=\int_0^T\int_U\mathds{1}_{(0,t]}(s)\Phi(s,x)\,\tilde N(\dl s,\dl x)$, $t\in[0,T]$.
A useful property shown in \cite[Lemma 3.1]{KovLinSch2015} is the following: There exists an isometric embedding $\kappa\colon L^2_{\pred}(\Omega\times[0,T];\cL_2(U_0,H))\to L^2_{\pred}(\Omega\times[0,T]\times U;H)$ such that 
$
\int_0^T\Phi(s)\,\dl L(s)=\int_0^T\int_U \Phi(s)x\,\tilde N(\dl s,\dl x)
$
for $\Phi\in L^2_{\pred}(\Omega\times [0,T];\cL_2(U_0,H))$, where we set $\Phi(s)x:=\kappa(\Phi)(s,x)$ to simplify notation.


\subsection{Poisson-Malliavin calculus in Hilbert space}
\label{subsec:Malliavin_prel}

In this subsection we collect some concepts and results from Hilbert space-valued Poisson Malliavin calculus. We refer to \cite{AnderssonLindner2017a} and the references therein for a more detailed exposition.

While in the Gaussian case the Malliavin derivative is a differential operator, one possible analogue in the Poisson case is a finite difference operator $D\colon L^0(\Omega;H)\to L^0(\Omega\times[0,T]\times U;H)$ defined as follows. 
Recall that $\mathcal{F}$ is the $\bP$-completion of the $\sigma$-algebra generated by the Lévy process $L$, which coincides with the $\bP$-completion of the $\sigma$-algebra generated by the Poisson random measure $N$. This and the factorization theorem from measure theory imply that for every random variable $F\colon \Omega\to H$ there exists a $\cN$-$\cB(H)$-measurable function $f\colon\mathbf N\to H$, called a representative of $F$, such that $F=f(N)$ $\bP$-almost surely. 
In this situation we set $\varepsilon^+_{t,x}F:=f(N+\delta_{(t,x)})$, where $\delta_{(t,x)}$ denotes Dirac measure at $(t,x)\in[0,T]\times U$.
As a consequence of Mecke's formula, this definition is $\bP\otimes\dl t\otimes\nu(\dl x)$-almost everywhere independent of the choice of the representative $f$, so that $F\mapsto\big(\varepsilon^+_{t,x}F\big)$ is well-defined as a mapping from $L^0(\Omega;H)$ to $L^0(\Omega\times[0,T]\times U;H)$, cf.~\cite[Lemma~2.5]{AnderssonLindner2017a}. 
The difference operator  $D\colon L^0(\Omega;H)\to L^0(\Omega\times[0,T]\times U;H),\;F\mapsto DF=\big(D_{t,x}F\big)$ is then defined by
\begin{align}
D_{t,x}F:=\varepsilon^+_{t,x}F-F,\quad (t,x)\in[0,T]\times U.
\end{align}
The Malliavin-Sobolev space $\mathds D^{1,2}(H)$ consists of all $F\in L^2(\Omega;H)$ satisfying $DF\in L^2(\Omega\times[0,T]\times U; H)$. In Subsection~\ref{subsec:regX} we introduce refined Malliavin-Sobolev spaces $\bM^{1,p,q}(H)$, $p,q\in(1,\infty]$.


The following basic lemmata are taken from \cite[Lemma 3.2 and Corollary 4.2]{AnderssonLindner2017a}.
\begin{lemma}\label{lem:chain}
Let $F\in L^0(\Omega;H)$ and $h$ be a measurable mapping from $H$ to another separable real Hilbert space $V$. Then it holds that  
$
  D h(F)
 =
  h(F+D F)-h(F).
$
\end{lemma}

\begin{lemma}\label{lem:DF_is_zero}
Let $t\in[0,T]$ and $F\colon\Omega\to H$ be $\cF_t$-$\cB(H)$-measurable.
Then the equality $D_{s,x}F=0$ holds $\bP\otimes\dl s\otimes\nu(\dl x)$-almost everywhere on $\Omega\times(t,T]\times U$. 
\end{lemma}

The next result is a special case of the general duality formula in \cite[Proposition~4.9]{AnderssonLindner2017a}. It is crucial for our approach to weak error analysis for Lévy driven SPDE.
\begin{prop}[Duality formula]\label{prop:duality}
For all $F\in\bD^{1,2}(H)$ and $\Phi\in L^{2}_{\pred}(\Omega\times[0,T]\times U;H)$ we have 
\begin{align*}
\bE\,\Big\langle F,\int_0^T\int_U\Phi(t,x)\,\tilde N(\dl t,\dl x)\Big\rangle=\E\int_0^T\int_U\big\langle D_{t,x}F,\,\Phi(t,x)\big\rangle\,\nu(\dl x)\,\dl t.
\end{align*}
\end{prop}

Before we proceed with two further important results, we need to discuss the application of $D$ on stochastic processes.
\begin{remark}[Difference operator for stochastic processes]
\label{rem:DXprocess}
One can define in a analogous way as above for stochastic processes a further difference operator $D$ mapping $X\in L^0(\Omega\times[0,T];H)$ to $DX=\big(D_{s,x}X(t)\big)_{t\in[0,T],(s,x)\in[0,T]\times U}\in L^0\big(\Omega\times[0,T]\times [0,T]\times U;H\big)$, see \cite[Remark~3.10]{AnderssonLindner2017a}.
Then it holds for $\lambda$-almost all $t\in[0,T]$ that
\begin{align}\label{eq:remDXprocess}
D_{s,x}X(t)=D_{s,x}(X(t))\quad\bP\otimes\dl s\otimes\nu(\dl x)\text{-a.e.},
\end{align} 
where $D(X(t))=\big(D_{s,x}(X(t))\big)_{(s,x)\in[0,T]\times U}\in L^0(\Omega\times[0,T]\times U;H)$ is for fixed $t$ the Malliavin derivative of the random variable $F=X(t)$ as introduced above. 
We will, however, typically encounter the situation where $X=(X(t))_{t\in[0,T]}$ is not given as an equivalence class of stochastic processes but as a single stochastic process with $X(t)$ being specifically defined for \emph{every} $t\in[0,T]$. If $X$ is not only $\cF\otimes\cB([0,T])$-measurable but also stochastically continuous or piecewise stochastically continuous, then there exists a $\bP\otimes\dl t\otimes\dl s\otimes\nu(\dl x)$-version of $DX=\big(D_{s,x}X(t)\big)_{t\in[0,T],(s,x)\in[0,T]\times U}$ such that \eqref{eq:remDXprocess}  holds for \emph{every} $t\in[0,T]$, cf.~\cite[Lemma~4.3]{AnderssonLindner2017a}. We also use a further analogously defined difference operator $D$ mapping $\Phi\in L^0(\Omega\times[0,T]\times U;H)$ to $D\Phi=\big(D_{s,x}\Phi(t,y)\big)_{(t,y),(s,x)\in[0,T]\times U}\in L^0(\Omega\times([0,T]\times U)^2;H)$ in such a way that for $\lambda\otimes\nu$-almost all $(t,y)\in[0,T]\times U$ we have $D_{s,x}\Phi(t,y)=D_{s,x}(\Phi(t,y))$ $\bP\otimes\dl s\otimes\nu(\dl x)$-a.e., cf.~\cite[Remark~3.10]{AnderssonLindner2017a}.
\end{remark}

In the regularity analysis of SPDEs it is important to know how $D$ acts on Lebesgue integrals and stochastic integrals. For this purpose we recall the following results. The first one is taken from \cite[Proposition 4.5]{AnderssonLindner2017a}, the second is a special case of \cite[Proposition 4.13]{AnderssonLindner2017a} combined with \cite[Lemma 4.11]{AnderssonLindner2017a}.

\begin{prop}[Malliavin derivative of time integrals]\label{lem:intDX}
Let $X\colon \Omega\times[0,T]\to H$ be a stochastic process which is $\cF\otimes\cB([0,T])$-measurable and piecewise stochastically continuous, let $\mu$ be a $\sigma$-finite Borel-measure on $[0,T]$, and assume that $X$ belongs to $L^1([0,T],\mu;L^p(\Omega;H))$ for some $p>1$. 
Consider a fixed 
version of
$DX=(D_{s,x}X(t))_{t\in[0,T],(s,x)\in [0,T]\times U}$ such that
for all $t\in[0,T]$ the identity $D_{s,x}X(t)=D_{s,x}(X(t))$ holds $\bP\otimes\dl s\otimes\nu(\dl x)$-almost everywhere, cf.~Remark~\ref{rem:DXprocess}.
Then, 
for all $B\in \cB(U)$ with $\nu(B)<\infty$
we have
$$\bE\Big[\int_{[0,T]}\int_B\int_{[0,T]}\|D_{s,x}X(t)\|\,\mu(\dl t)\,\nu(\dl x)\,\dl s\Big]<\infty,$$
so that the integral $\int_{[0,T]} D_{s,x}X(t)\,\mu(\dl t)$ is defined $\bP\otimes\dl s\otimes\nu(\dl x)$-almost everywhere on $\Omega\times[0,T]\times U$ as an $H$-valued Bochner integral. Moreover,
the equality
\[D_{s,x}\int_{[0,T]} X(t)\,\mu(\dl t)=\int_{[0,T]} D_{s,x}X(t)\,\mu(\dl t)\]
holds $\bP\otimes\dl s\otimes\nu(\dl x)$-almost everywhere on $\Omega\times[0,T]\times U$.  
\end{prop}

\begin{prop}[Malliavin derivative of stochastic integrals]
\label{prop:comm_space-time2}
Let $\Phi\in L^{2}_{\pred}(\Omega\times[0,T]\times U;H)$. Then the derivative 
$D\Phi\in L^0\big(\Omega\times([0,T]\times U)^2;H\big)$ has a $\cP_T\otimes\cB(U)\otimes\cB([0,T]\times U)$-measurable version,
i.e., the mapping
\begin{equation*}\label{eq:space_time2}
D\Phi\colon\Omega\times([0,T]\times U)^2\to H,\;(\omega,t,y,s,x)\mapsto D_{s,x}\Phi(\omega,t,y)
\end{equation*} 
has a $\P\otimes(\lambda\otimes\nu)^{\otimes2}$-version which is $\cP_T\otimes\cB(U)\otimes\cB([0,T]\times U)$-measurable.
If moreover $\E\int_0^T\int_U\|D_{s,x}\Phi(t,y)\|^2\,\nu(\dl y)\,\dl t<\infty$ for $\lambda\otimes\nu$-almost all $(s,x)\in[0,T]\times U$, then
the equality
\begin{align}\label{eq:comm_rel_space-time}
D_{s,x}\int_0^T\int_U\Phi(t,y)\,\tilde N(\dl t,\dl y)=\int_0^T\int_UD_{s,x}\Phi(t,y)\,\tilde N(\dl t,\dl y)+\Phi(s,x)
\end{align}
holds 
$\bP\otimes\dl s\otimes\nu(\dl x)$-almost everywhere
on $\Omega\times[0,T]\times U$.  
\end{prop}

\section{Malliavin regularity for a class of semilinear SPDE}
\label{sec:reg}


\subsection{Assumptions on the considered equation}
\label{subsec:assumptionSPDE}

We next state the precise assumptions on the operator $A$, the driving noise $L$, the nonlinearity $F$, and the initial value $X_0$ in Eq.~\eqref{eq:SPDE_additive}.

\begin{assumption}\label{as:SPDE} 
In addition to Assumption~\ref{setting}, suppose that the following holds:
\begin{itemize}[leftmargin=8mm]
\item[(i)]
\label{as:A}
The operator $A\colon D(A)\subset H\to H$ is densely defined, linear, self-adjoint, positive definite and has a compact inverse. 
In particular, $-A$ is the generator of an analytic semigroup of contractions, which we denote by $(S(t))_{t\geq0}\subset \cL(H)$.
The spaces $\dot H^\rho$, $\rho\in\R$, are defined for $\rho\geq0$ as $\dot H^\rho:= D(A^{\frac{\rho}2})$ with norm 
$\|\cdot\|_{\dot H^\rho}:=\|A^{\frac{\rho}2}\cdot\|$ and for $\rho<0$ as the closure of $H$ w.r.t.\ the analogously defined
$\|\cdot\|_{\dot H^\rho}$-norm.  
\item[(ii)]
\label{as:L2_WC}
For some $\beta\in(0,1]$, the state space $U$ of the Lévy process $L=(L(t))_{t\in[0,T]}$ in Assumption~\ref{setting} is given by $U=\dot H^{\beta-1}$. 
\item[(iii)]
\label{as:F_WC}
For some $\delta\in[1-\beta,2)$,
the drift function $F\colon H\to \dot H^{\beta-1}$ belongs to the class $\Lip^0(H,\dot H^{\beta-1})\cap\Lip^1(H,\dot H^{-\delta})$.
\item[(iv)]
\label{as:X0_WC}
The initial value $X_0$ is an element of the space $\dot H^{2\beta}$.
\end{itemize}
\end{assumption}


It is well known that, under Assumption~\ref{as:A}(i), there exist constants $C_\rho\in[0,\infty)$ (independent of $t$) such that
\begin{align}
  \big\|A^\frac{\rho}2S(t)\big\|_{\LB(H)}
&\leq
  C_\rho\, t^{-\frac{\rho}2},
  \quad
  t>0, \ \rho\geq0,\label{eq:smoothing}\\
  \big\|A^{-\frac{\rho}2}(S(t)-\id_H)\big\|_{\LB(H)}
&\leq
  C_\rho \, t^{\frac{\rho}2},
  \quad \;\;
  t\geq0, \ \rho\in(0,2],\label{eq:continuity}
\end{align}
see, e.g.,~\cite[Section~2.6]{pazy1983}. 
Concerning Assumption~\ref{as:SPDE}(iii), let us remark that Lipschitz continuity of the derivative $F'$ of $F$ is needed for the weak convergence analysis in Section~\ref{sec:weak}. 
Assuming
$F\in \Lip^1(H,H)$ is sufficient for the analysis, compare, e.g., \cite{AnderssonKruseLarsson}. In applications to SPDE this assumption is not satisfactory as the most important type of nonlinear drift, the Nemytskii type drift, typically does not satisfy the assumption. 
By assuming that $F'$ is Lipschitz continuous only as a mapping into the larger space $\dot H^{-\delta}$, for suitable $\delta$, the Sobolev embedding theorem can be used to prove that Nemytskii type nonlinearities are in fact included in $d\in\{1,2,3\}$ space dimensions. 
More precisely this holds for $\delta>\frac{d}2$, compare \cite[Example 3.2]{Wang2016}. 

\begin{example}\label{ex:A}
For $d\in\{1,2,3\}$ let $\cO\subset \bR^d$ be an open, bounded, convex, poly-gonal/polyhedral domain and set $H:=L^2(\cO)$. 
Our standard example for $A$ is a second order elliptic partial differential operator with zero Dirichlet boundary condition of the form
$
Au:=-\nabla\cdot(a\nabla u)+cu,\; u\in D(A):=H^1_0(\cO)\cap H^2(\cO),
$
with bounded and sufficiently smooth coefficients $a,c\colon\cO\to\bR$ such that $a(\xi)\geq\theta>0$ and $c(\xi)\geq 0$ for all $\xi\in\cO$. Here $H^1_0(\cO)$ and $H^2(\cO)$ are the classical $L^2$-Sobolev spaces of order one with zero Dirichlet boundary condition and of order two, respectively. 
As an example for the drift function $F$ we consider the Nemytskii type nonlinearity given by $(F(x))(\xi)=f(x(\xi))$, $x\in L^2(\cO)$, $\xi\in \cO$, where $f\colon\bR\to\bR$ is twice continuously differentiable with bounded first and second derivative. In this situation, Assumption~\ref{as:SPDE}(iii) is fullfilled for $\delta>\frac d2$, compare \cite[Example 3.2]{Wang2016}. 
Concrete examples for the Lévy process $L$ can be found in \cite[Subsection~2.1]{KovLinSch2015}.
\end{example}






By a mild solution to Eq.~\eqref{eq:SPDE_additive} we mean an $(\mathcal F_t)_{t\in[0,T]}$-predictable stochastic process 
$X\colon\Omega\times[0,T] \to H$ such that 
\begin{align}\label{eq:XL2}
\sup_{t\in[0,T]}\|X(t)\|_{L^2(\Omega;H)} < \infty,
\end{align} 
and such that for all $t\in[0,T]$ it holds $\P$-almost surely that 
\begin{align}\label{eq:SPDE_mild}
  X(t)
  =
  S(t)X_0
  +
  \int_0^t
    S(t-s)
    F(X(s))
  \diffout s
  +
  \int_0^t
    S(t-s)
  \diffout L(s).
\end{align}
Under Assumption~\ref{as:SPDE} there exists a unique (up to modification) mild solution $X$ to Eq.~\eqref{eq:SPDE_additive}. This follows, e.g., from a straightforward modification of the proof of \cite[Theorem~9.29]{PesZab2007}, where slightly different assumptions are used. Moreover, this solution is mean-square continuous, i.e., $X\in\cC([0,T],L^2(\Omega;H))$, which can be seen by using standard arguments analogous to those used in the Gaussian case.

\subsection{Regularity results for the solution process}
\label{subsec:regX}
We are now ready to analyze the Malliavin regularity of the mild solution to Eq.~\eqref{eq:SPDE_additive}.

\begin{prop}\label{prop:DX}
Let Assumption~\ref{as:SPDE} hold, 
let $X=(X(t))_{t\in[0,T]}$ be the mild solution to Eq.~\eqref{eq:SPDE_additive}, 
and consider a fixed 
version of
$DX=(D_{s,x}X(t))_{t\in[0,T],(s,x)\in [0,T]\times U}$ such that
for all $t\in[0,T]$ the identity $D_{s,x}X(t)=D_{s,x}(X(t))$ holds $\bP\otimes\dl s\otimes\nu(\dl x)$-almost everywhere, cf.~Remark~\ref{rem:DXprocess}.
Then for all $t\in[0,T]$ and all $B\in\cB(U)$ with $\nu(B)<\infty$ we have
\begin{align}\label{eq:DX1}
\bE\int_0^T\int_B\int_0^t \Big\|S(t-r)\Big[F\big(X(r)+D_{s,x}X(r)\big)-F\big(X(r)\big)\Big]\Big\|\,\dl r\,\nu(\dl x)\,\dl s<\infty,
\end{align}
so that for all $t\in[0,T]$ the integral $\int_0^tS(t-r)\big[F\big(X(r)+D_{s,x}X(r)\big)-F\big(X(r)\big)\big]\,\dl r$ is defined 
$\bP\otimes\dl s\otimes\nu(\dl x)$-almost everywhere 
on $\Omega\times[0,T]\times U$ as an $H$-valued Bochner integral.
Moreover, for all $t\in[0,T]$ the equality \eqref{eq:DX2}
holds $\bP\otimes\dl s\otimes\nu(\dl x)$-almost everywhere on $\Omega\times[0,T]\times U$.
\end{prop}

\begin{proof}
We fix $t\in[0,T]$ and apply the difference operator $D\colon L^0(\Omega;H)\to L^0(\Omega\times[0,T]\times U;H)$ to the single terms in \eqref{eq:SPDE_mild}.
As the initial value $X_0$ is deterministic, it is clear that $D_{s,x}(S(t)X_0)=0$ $\bP\otimes\dl s\otimes\nu(\dl x)$-almost everywhere on $\Omega\times[0,T]\times U$.
Next, observe that by \eqref{eq:smoothing}, the linear growth of $F$ and  \eqref{eq:XL2} we have
\begin{equation}\label{eq:DX3}
\begin{aligned}
&\int_0^t\big\|S(t-r)F(X(r))\big\|_{L^2(\Omega;H)}\dl r\\
&\leq C_{1-\beta}\|F\|_{\Lip(H,\dot H^{\beta-1})}\int_0^t (t-r)^{\frac{\beta-1}2}\big(1+\|X(r)\|_{L^2(\Omega;H)}\big)\,\dl r\\
&\leq C_{1-\beta}\|F\|_{\Lip(H,\dot H^{\beta-1})}\frac{2}{\beta+1}T^{\frac{\beta+1}2}\,\big(1+\sup_{t\in[0,T]}\|X(t)\|_{L^2(\Omega;H)}\big)\,<\infty.
\end{aligned}
\end{equation} 
Proposition~\ref{lem:intDX} thus implies 
\eqref{eq:DX1} 
and that the equality
$
D_{s,x}\int_0^t S(t-r)F(X(r))\,\dl r=\int_0^t D_{s,x}\big(S(t-r)F(X(r))\big)\,\dl r
$
holds
$\bP\otimes\dl s\otimes\nu(\dl x)$-a.e.\ on $\Omega\times[0,T]\times U$. 
Hereby we consider a  version of $\big(D_{s,x}\big(S(t-r)F(X(r))\big)\big)_{r\in[0,t),(s,x)\in[0,T]\times U}$ which is $\cF\otimes\cB([0,t))\otimes\cB([0,T]\times U)$-measurable, cf.~Remark~\ref{rem:DXprocess}. Using also Lemma~\ref{lem:chain} and Lemma~\ref{lem:DF_is_zero},
 we obtain
\begin{equation*}
\begin{aligned}
&D_{s,x}\int_0^t S(t-r)F(X(r))\,\dl r\\
&=\mathds{1}_{s\leq t}\cdot\int_s^tS(t-r)\big[F\big(X(r)+D_{s,x}X(r)\big)-F\big(X(r)\big)\big]\,\dl r
\end{aligned}
\end{equation*}
$\bP\otimes\dl s\otimes\nu(\dl x)$-a.e.\ on $\Omega\times[0,T]\times U$.
Finally, the identity $\int_0^t S(t-r)\,\dl L(r)=\int_0^t\int_U S(t-r)x\,\tilde N(\dl r,\dl x)$ and the commutation relation in Proposition~\ref{prop:comm_space-time2} yield
\begin{align*}
D_{s,x}\int_0^t S(t-r)\,\dl L(r)=\mathds{1}_{s\leq t}\cdot S(t-s)x
\end{align*}
$\bP\otimes\dl s\otimes\nu(\dl x)$-a.e.\ on $\Omega\times[0,T]\times U$.
Summing up, we have shown that \eqref{eq:DX2} holds for every fixed $t\in[0,T]$ as an equality in $L^0(\Omega\times[0,T]\times U;H)$.
%
\end{proof}

The refined Malliavin-Sobolev spaces introduced next and the subsequent regularity results have Gaussian counterparts in \cite{AnderssonKovacsLarsson,AnderssonKruseLarsson}.

\begin{definition}[Refined Sobolev-Malliavin spaces]
Consider the setting described in Subsections~\ref{subsec:square} and \ref{subsec:Malliavin_prel}.
For $p,q\in(1,\infty]$ we define $\bM^{1,p,q}(H)$ as the space consisting of all $F\in L^p(\Omega;H)$ such that $DF\in L^p(\Omega;L^q([0,T];L^2(U;H)))$. It is equipped with the seminorm 
$
  |F|_{\bM^{1,p,q}(H)}
 : =
  \|DF\|_{L^p(\Omega;L^q([0,T];L^2(U;H)))}
$ and norm
\begin{align*}
  \|F\|_{\bM^{1,p,q}(H)}
  :=
  \Big(
    \|F\|_{L^p(\Omega;H)}^p
    +
    |F|_{\bM^{1,p,q}(H)}^p
  \Big)^\frac1p.
\end{align*}
For $p,p',q,q'\in(1,\infty)$ such that $\frac1{p}+\frac{1}{p'}=\frac1{q}+\frac{1}{q'}=1$, the space 
$\bM^{-1,p',q'}(H)$ is defined as the (topological) dual space of $\bM^{1,p,q}(H)$.
\end{definition}

Arguing as in \cite[Proposition~3.7]{AnderssonLindner2017a} one finds that  $\bM^{1,p,q}(H)$ is a Banach space for all $p,q\in(1,\infty)$. If additionally $p\in[2,\infty)$, then $\bM^{1,p,q}(H)$ is continuously embedded in $L^2(\Omega;H)$. This embedding is dense according to \cite[Lemma~3.8]{AnderssonLindner2017a}. In this situation we will use the Gelfand triple $\bM^{1,p,q}(H)\subset L^2(\Omega;H)\subset \bM^{-1,p',q'}(H)$.


\label{subsec:weak1}
\begin{prop}[Regularity I]\label{thm:reg2}
Let Assumption~\ref{as:SPDE}
hold. 
Depending on the value of $\beta\in(0,1]$, we assume either that $q\in(1,\tfrac2{1-\beta})$ if $\beta\in(0,1)$ or $q=\infty$ if $\beta=1$.
Then it holds that
\begin{align}\label{eq:additive_reg}
  \sup_{t\in[0,T]}|X(t)|_{\bM^{1,\infty,q}(H)}<\infty.
\end{align}
As a consequence, we also have $\sup_{t\in[0,T]}\|X(t)\|_{\bM^{1,2,q}(H)}<\infty$.
\end{prop}

\begin{proof}
We consider a fixed 
version of
$DX=(D_{s,x}X(t))_{t\in[0,T],(s,x)\in [0,T]\times U}$ such that
for all $t\in[0,T]$ the identity $D_{s,x}X(t)=D_{s,x}(X(t))$ holds $\bP\otimes\dl s\otimes\nu(\dl x)$-almost everywhere, cf.~Remark~\ref{rem:DXprocess}.
As a consequence of Proposition~\ref{prop:DX}, the smoothing property \eqref{eq:smoothing}, the fact that $U=\dot H^{\beta-1}$ and the Lipschitz continuity of $F$, we know that for all $t\in[0,T]$ the estimate
\begin{equation}\label{eq:regI1}
\begin{aligned}
\|D_{s,x}X(t)\|
&\leq
\mathds{1}_{s\leq t}\cdot C_{1-\beta}|F|_{\Lip^0(H,\dot H^{\beta-1})}
\int_s^t(t-r)^{\frac{\beta-1}2}\|D_{s,x}X(r)\|\,\dl r\\
&\quad+\mathds{1}_{s\leq t}\cdot C_{1-\beta}\|x\|_{U}(t-s)^{\frac{\beta-1}2}
\end{aligned}
\end{equation}
holds $\bP\otimes\dl s\otimes\nu(\dl x)$-almost everywhere on $\Omega\times[0,T]\times U$.
Moreover, Proposition~\ref{lem:intDX} and $\eqref{eq:XL2}$ imply that
\begin{align}\label{eq:regI2}
\int_0^T\|D_{s,x}X(t)\|\,\dl t<\infty
\end{align}
$\bP\otimes\dl s\otimes\nu(\dl x)$-almost everywhere on $\Omega\times[0,T]\times U$.

In order to be able to apply the generalized Gronwall Lemma~\ref{lem:gronwall}, we construct a new version of $(D_{s,x}X(t))_{t\in[0,T],(s,x)\in [0,T]\times U}$ such that the estimates \eqref{eq:regI1} and \eqref{eq:regI2} hold \emph{everywhere} on $\Omega\times[0,T]\times[0,T]\times U$ and $\Omega\times[0,T]\times U$, respectively.
For this purpose, let $A\in\cF\otimes\cB([0,T])\otimes\cB([0,T]\times U)$ be 
the set consisting of all $(\omega,t,s,x)\in\Omega\times[0,T]\times[0,T]\times U$ for which 
\eqref{eq:regI1} holds.
Let $B\in\cF\otimes\cB([0,T]\times U)$
be the set consisting of all $(\omega,s,x)\in\Omega\times[0,T]\times U$ for which \eqref{eq:regI1} holds $\dl t$-almost everywhere on $[0,T]$. 
Finally, let $C\in\cF\otimes\cB([0,T]\times U)$
be the set consisting of all $(\omega,s,x)\in\Omega\times[0,T]\times U$ for which \eqref{eq:regI2} holds. 
Let $\Gamma\colon \Omega\times[0,T]\times[0,T]\times U\to H$ be defined by $\Gamma:=\mathds{1}_{A\cap \pi^{-1}(B\cap C)}DX$, where $\pi\colon \Omega\times[0,T]\times[0,T]\times U\to \Omega\times[0,T]\times U$ is the coordinate projection given by $\pi(\omega,t,s,x):=(\omega,s,x)$.
Note that for all $t\in[0,T]$ the identity $\Gamma(\cdot,t,\cdot,\cdot)=DX(t)$ holds $\bP\otimes\lambda\otimes\nu$-almost everywhere on $\Omega\times[0,T]\times U$.
 We choose $\Gamma$ as our new version of $DX$ and henceforth write $DX=(D_{s,x}X(t))_{t\in[0,T],(s,x)\in [0,T]\times U}$ instead of $\Gamma$ to simplify notation. Observe that for this new version the estimates \eqref{eq:regI1} and \eqref{eq:regI2} hold indeed everywhere on $\Omega\times[0,T]\times[0,T]\times U$ and $\Omega\times[0,T]\times U$, respectively.
The generalized Gronwall Lemma~\ref{lem:gronwall} thus implies that there exists a constant $C=C\big(C_{1-\beta}|F|_{\Lip^0(H,\dot H^{\beta-1})},T,\beta\big)\in[0,\infty)$ such that the estimate
\begin{align}\label{eq:regI3}
\|D_{s,x}X(t)\|\leq \mathds{1}_{s\leq t}\,C\,C_{1-\beta}\|x\|_U(t-s)^{\frac{\beta-1}2}
\end{align}
holds everywhere on $\Omega\times[0,T]\times[0,T]\times U$.

Assume the case where $\beta\in(0,1)$, $q\in (1,\frac2{1-\beta})$ and consider the version of $DX=(D_{s,x}X(t))_{t\in[0,T],(s,x)\in [0,T]\times U}$ constructed above. Integration of \eqref{eq:regI3} yields 
\begin{align*}
&\sup_{t\in[0,T]}\Big[\int_0^T\Big(\int_U\|D_{s,x}X(t)\|^2\nu(\dl x)\Big)^{\frac q2}\,\dl s\Big]^{\frac1q}\\
&\leq
C\,C_{1-\beta}\,|\nu|_2\sup_{t\in[0,T]}\Big[\int_0^t(t-s)^{q\cdot\frac{\beta-1}2}\,\dl s\Big]^{\frac1q}\leq
C\,C_{1-\beta}\,|\nu|_2\frac{1}{(q\cdot\frac{\beta-1}2+1)^{\frac1q}}T^{\frac{\beta-1}2+\frac1q},
\end{align*}
which implies \eqref{eq:additive_reg}. 
The case where $\beta=1$ and $q=\infty$ is treated similarly. Finally, the second assertion in Proposition~\ref{thm:reg2} follows from \eqref{eq:additive_reg} and \eqref{eq:XL2}. 
\end{proof}


\begin{prop}[Negative norm inequality]\label{prop:negnorm2}
Consider the setting described in \linebreak Subsections~\ref{subsec:square} and \ref{subsec:Malliavin_prel}.
Let $p',q'\in(1,2]$. For predictable integrands 
$\Phi\in L_{\mathrm{pr}}^2(\Omega\times[0,T]\times U;H)$ it holds that
\begin{align*}
  \Big\|
    \int_0^T \int_U
      \Phi(t,y)
    \tilde N(\diffin t, \diffin y)
  \Big\|_{\bM^{-1,p',q'}(H)}
  \leq
  \|\Phi\|_{L^{p'}(\Omega;L^{q'}([0,T];L^2(U;H)))}.
\end{align*}
\end{prop}

\begin{proof}
Let $p,q\in[2,\infty)$ satisfy $\tfrac1p+\tfrac1{p'}=\tfrac1q+\tfrac1{q'}=1$. By the duality formula from Proposition~\ref{prop:duality}, duality in the Gelfand triple $\bM^{1,p,q}(H)
  \subset L^2(\Omega;H)
  \subset
  \bM^{-1,p',q'}(H)$, and by the H\"older inequality it holds that
\begin{align*}
&\Big\|\int_0^T \int_U
      \Phi(t,y)
    \tilde N(\diffin t, \diffin y)\Big\|_{\bM^{-1,p',q'}(H)}\\
&\quad
  =
  \sup_{Z\in\bM^{1,p,q}(H)\setminus\{0\}}
  \frac{
    \big\langle Z,\int_0^T \int_U
      \Phi(t,y)
    \tilde N(\diffin t, \diffin y)\big\rangle_{L^2(\Omega;H)}
  }{
    \|Z\|_{\bM^{1,p,q}(H)}
  }\\
&\quad
  =
  \sup_{Z\in\bM^{1,p,q}(H)\setminus\{0\}}
  \frac{
    \langle D Z,\Phi \rangle_{L^2(\Omega\times[0,T]\times U;H)}
  }{
    \|Z\|_{\bM^{1,p,q}(H)}
  }\\
&\quad
  \leq
  \sup_{Z\in\bM^{1,p,q}(H)\setminus\{0\}}
  \frac{
    \| D Z\|_{L^{p}(\Omega;L^{q}([0,T];L^2(U;H)))}
    \|\Phi \|_{L^{p'}(\Omega;L^{q'}([0,T];L^2(U;H)))}
  }{
    \|Z\|_{\bM^{1,p,q}(H)}
  }\\
&\quad
  \leq
  \|\Phi\|_{L^{p'}(\Omega;L^{q'}([0,T];L^2(U;H)))}.\qedhere
\end{align*}
\end{proof}


\begin{prop}[Regularity II]\label{thm:reg3}
Let Assumption~\ref{as:SPDE} hold 
and $X=(X(t))_{t\in[0,T]}$ be the mild solution to Eq.~\eqref{eq:SPDE_additive}.  For all $\gamma\in[0,\beta)$ and $q'=\tfrac2{1+\gamma}$ there exist a constant 
$C\in[0,\infty)$
such that
\begin{align*}
  \|X(t_2)-X(t_1)\|_{\bM^{-1,2,q'}(H)}
  \leq
  C |t_2 - t_1|^{\gamma},
  \quad
  t_1,t_2\in [0,T].
\end{align*}
\end{prop}

\begin{proof}
Let $0\leq t_1\leq t_2\leq T$. Representing the increment $X(t_2)-X(t_1)$ via \eqref{eq:SPDE_mild}, 
taking norms and using the continuous embedding $L^2(\Omega;H)\subset \bM^{-1,2,q'}(H)$, we obtain
\begin{align*}
&\|
    X(t_2) - X(t_1) 
  \|_{\bM^{-1,2,q'}(H)}
  \leq
  \|
    (S(t_2-t_1)-\id_H)A^{-\gamma}S(t_1) A^{\gamma}X_0
  \|\\
&\quad
  +
  \Big\|
    \int_0^{t_1}
      (S(t_2-t_1)-\id_H)A^{-\gamma}A^\gamma S(t_1-s)A^{\frac{1-\beta}2}A^{\frac{\beta-1}2}
      F(X(s))
    \diffout s
  \Big\|_{L^2(\Omega;H)}\\
&\quad
  +
  \Big\|
    \int_{t_1}^{t_2}
      S(t_2-s))A^{\frac{1-\beta}2}A^{\frac{\beta-1}2}F(X(s))
    \diffout s
  \Big\|_{L^2(\Omega;H)}\\
&\quad
  +
  \Big\|
    \int_0^{t_1}\int_{\dot H^{\beta-1}}
      (S(t_2-t_1)-\id_H)A^{-\gamma}A^{\gamma}S(t_1-s)x
      \,\tilde N(\dl s,\dl x)  
  \Big\|_{\bM^{-1,2,q'}(H)}\\
&\quad
  +
  \Big\|
    \int_{t_1}^{t_2}\int_{\dot H^{\beta-1}}
      S(t_2-s)A^{\frac{1-\beta}2}A^{\frac{\beta-1}2}x
    \,\tilde N(\dl s,\dl x)
  \Big\|_{\bM^{-1,2,q'}(H)}.
\end{align*}
Further, from \eqref{eq:smoothing}, \eqref{eq:continuity}, \eqref{eq:XL2}, the linear growth of $F$ and the negative norm inequality in Proposition~\ref{prop:negnorm2} we obtain
\begin{equation}\label{eq:long}
\begin{split}
&\|
    X(t_2) - X(t_1) 
  \|_{\bM^{-1,2,q'}(H)}
  \,\leq\,
  C_0 C_{2\gamma}
  \|
    X_0
  \|_{\dot H^{2\gamma}}(t_2-t_1)^\gamma\\
&
  +
  C_{2\gamma+1-\beta}
  \int_0^{t_1}(t_1-s)^{-\gamma-\frac{1-\beta}2}\dl s\,
  C_{2\gamma}\,(t_2-t_1)^\gamma\\
&\quad\cdot\|F\|_{\Lip^0(H,\dot H^{\beta-1})}
  \big(1+\sup_{t\in[0,T]}
    \|X(t)\|_{L^2(\Omega;H)}\big)\\
&
  +
  C_{1-\beta}\int_{t_1}^{t_2}(t_2-s)^{-\frac{1-\beta}2}\,\dl s\,
    \|F\|_{\Lip^0(H,\dot H^{\beta-1})}
  \big(1+\sup_{t\in[0,T]}
    \|X(t)\|_{L^2(\Omega;H)}\big)\\
&
  +
  \Big[
    \int_0^{t_1}\Big(\int_{\dot H^{\beta-1}}
      \big\|
        (S(t_2-t_1)-\id_H)A^{-\gamma}A^{\gamma}S(t_1-s)
        A^{\frac{1-\beta}2}A^{\frac{\beta-1}2}x
      \big\|^2\nu(\dl x)
      \Big)^{\frac{q'}2}
    \diffout s
  \Big]^{\frac1{q'}}\\
&
  +
  \Big[
    \int_{t_1}^{t_2}\Big(\int_{\dot H^{\beta-1}}
      \big\|
        S(t_2-s)A^{\frac{1-\beta}2}A^{\frac{\beta-1}2}x
      \big\|^2\nu(\dl x)
      \Big)^{\frac{q'}2}
    \diffout s
  \Big]^{\frac1{q'}}.
\end{split}
\end{equation}
Note that $\gamma+\frac{1-\beta}2<1$ and thus the integral in the second term on the right hand side of \eqref{eq:long} is bounded by $(-\gamma+\frac{1+\beta}2)^{-1}T^{-\gamma+\frac{1+\beta}2}$.
The integral in the third term on the right hand side of \eqref{eq:long} satisfies
$\int_{t_1}^{t_2}(t_2-s)^{-\frac{1-\beta}2}\,\dl s
=\frac{2}{1+\beta}(t_2-t_1)^{\frac{1+\beta}2}
\lesssim (t_2-t_1)^\gamma.$
The fourth term on the right hand side of \eqref{eq:long} can be estimated by
$
  C_{2\gamma}\,(t_2-t_1)^\gamma\,
  C_{2\gamma + (1-\beta)}
  \big(
    \int_0^{t_1}
      (t_1-s)^{-\frac2{1+\gamma} \frac{2\gamma+1-\beta}2}
    \diffout s
  \big)^{\frac{1+\gamma}2}\,|\nu|_2
  \lesssim (t_2-t_1)^\gamma.
$
The latter integral is bounded by $\int_0^{T}
      s^{-\frac2{1+\gamma} \frac{2\gamma+1-\beta}2}
    \diffout s$, which is finite since
\begin{align}\label{eq:par_calc}
  \frac2{1+\gamma} \frac{2\gamma+1-\beta}2
  =
  \frac{1 + \gamma -(\beta-\gamma)}{1+\gamma}
  <1.
\end{align}
Finally, the the last term \eqref{eq:long} is bounded by
$
  C_{1-\beta}\,
  \big(
    \int_{t_1}^{t_2}
      (t_2-s)^{\frac2{1+\gamma}\frac{\beta-1}2}    
    \diffout s
  \big)^{\frac{1+\gamma}2}|\nu|_2
  =
  \big(\frac{\beta+\gamma}{1+\gamma}\big)^{-\frac{1+\gamma}2}C_{1-\beta}|\nu|_2
  (
    t_2-t_1
  )^{\frac{\beta+\gamma}2}
  \lesssim (t_2-t_1)^\gamma.
$
This completes the proof.
\end{proof}

\subsection{Auxiliary results on refined Malliavin Sobolev spaces}
\label{subsec:auxiliary}
In the sequel, we consider the setting described in Subsection~\ref{subsec:square} and \ref{subsec:Malliavin_prel}.

\begin{lemma}\label{lem:op_norm}
Let $p,q\in(1,\infty)$, let $V_1$, $V_2$ be separable real Hilbert spaces, and let
$\varphi\colon H\to \cL(V_1,V_2)$ be a bounded function belonging to the class $\Lip^0(H,\cL(V_1,V_2))$. 
For all $Y\in L^0(\Omega;H)$ satisfying 
$DY\in L^\infty(\Omega;L^q([0,T];L^2(U;H)))$ and all $Z\in \bM^{1,p,q}(V_1)$ it holds that $\varphi(Y)Z\in \bM^{1,p,q}(V_2)$ and
\begin{align*}
  \|\varphi(Y)Z\|_{\bM^{1,p,q}(V_2)}
&\leq
  2^{\frac1p}\Big(2\sup_{x\in H}\|\varphi(x)\|_{\cL(V_1,V_2)} + |\varphi|_{\Lip^0(H,\cL(V_1,V_2))}|Y|_{\bM^{1,\infty,q}(H)}\Big)\\
&\quad\cdot
  \|Z\|_{\bM^{1,p,q}(V_1)}.
\end{align*}
\end{lemma}

\begin{proof}
Take $Y$ and $Z$ as in the statement and observe that
\begin{equation}\label{eq:lem_opnorm_1}
  \|\varphi(Y)Z\|_{L^{p}(\Omega;V_2)}
  \leq
  \sup_{x\in H}\|\varphi(x)\|_{\cL(V_1,V_2)}\|Z\|_{L^p(\Omega;V_1)}.
\end{equation}
Next, due to the definition of the difference operator $D$ in  Subsection~\ref{subsec:Malliavin_prel} and  the identities $D_{t,y}Y = \varepsilon_{t,y}^+Y-Y$ and $\varepsilon_{t,y}^+Y=Y + D_{t,y}Y$ it holds $\bP\otimes\dl t\otimes\nu(\dl y)$--almost everywhere on $\Omega\times [0,T]\times U$ that
$
D_{t,y}(\varphi(Y)Z)
=
  \varphi(\varepsilon_{t,y}^+Y)\,\varepsilon_{t,y}^+Z - \varphi(Y)Z=
  \varphi(Y+D_{t,y}Y)\,D_{t,y}Z + (\varphi(Y + D_{t,y}Y)-\varphi(Y))Z.
$
As a consequence, we obtain 
\begin{equation}\label{eq:lem_opnorm_2}
\begin{aligned}
 \big\|D\big(\varphi(Y)&Z\big)\big\|_{L^p(\Omega;L^q([0,T];L^2(U;V_2)))}
  \leq
  \Big(
    \sup_{x\in H}\|\varphi(x)\|_{\cL(V_1,V_2)}\\
&\quad
    +    
    |\varphi|_{\Lip^0(H,\cL(V_1,V_2))}
    \|DY\|_{L^{\infty}(\Omega;L^q([0,T];L^2(U;H)))}
  \Big)
  \|Z\|_{\bM^{1,p,q}(V_1)}.
\end{aligned}
\end{equation}
Combining \eqref{eq:lem_opnorm_1} and \eqref{eq:lem_opnorm_2} finishes the proof.
\end{proof}

\begin{prop}[Local Lipschitz bound]\label{prop:Lip2}
Let $p',q'\in(1,2]$, $q\in[2,\infty)$ be such that $\frac1{q}+\frac{1}{q'}=1$, let $V$ be a separable real Hilbert space, and 
$\psi\in \Lip^1(H,V)$. 
Then there exists $C\in[0,\infty)$ such that for all $Y_1,Y_2\in L^2(\Omega;H)$ with $DY_1,DY_2\in L^\infty(\Omega;L^{q}([0,T];L^2(U;H)))$ it holds that
\begin{align*}
&\|
    \psi(Y_1)-\psi(Y_2)
  \|_{\bM^{-1,p',q'}(V)}\\
&\quad
  \leq
  4
  \Big(
    |\psi|_{\Lip^0(H,V)}
    +
    |\psi|_{\Lip^1(H,V)}
    \sum_{i=1}^2|Y_i|_{\bM^{1,\infty,q}(H)}
  \Big)
  \|Y_1-Y_2\|_{\bM^{-1,p',q'}(H)}.
\end{align*}
\end{prop}

\begin{proof}
Let $p=p'/(p'-1)$. Due to the fundamental theorem of calculus and the structure of the Gelfand triple $\bM^{1,p,q}(V)\subset L^2(\Omega;V)\subset \bM^{-1,p',q'}(V)$, it holds that
\begin{align*}
  &\|\psi(Y_1)-\psi(Y_2)\|_{\bM^{-1,p',q'}(V)}
  =\Big\|\int_0^1\psi'\big(Y_2+\lambda(Y_1-Y_2)\big)(Y_1-Y_2)\,\dl\lambda\Big\|_{\bM^{-1,p',q'}(V)}\\
  &=
  \sup_{\substack{Z\in \bM^{1,p,q}(V)\\ \|Z\|_{\bM^{1,p,q}(V)}=1}}
    \Big\langle
      Z, \int_0^1\psi'\big(Y_2+\lambda(Y_1-Y_2)\big)(Y_1-Y_2)\,\dl\lambda
    \Big\rangle_{L^2(\Omega;V)}\\
  &\leq
  \sup_{\substack{Z\in \bM^{1,p,q}(V)\\ \|Z\|_{\bM^{1,p,q}(V)}=1}}
  \int_0^1\big\|\big[\psi'\big(Y_2+\lambda(Y_1-Y_2)\big)\big]^*Z\big\|_{\bM^{1,p,q}(H)}\,\dl\lambda\,\|Y_1-Y_2\|_{\bM^{-1,p',q'}(H)},
\end{align*}
where for $x\in H$ we denote by $[\psi'(x)]^*\in\cL(V,H)$ the Hilbert space adjoint of $\psi'(x)\in\cL(H,V)$. 
Note that the mapping $\varphi\colon H\to \cL(V,H)$ defined by $\varphi(x):=[\psi'(x)]^*$, $x\in H$, is bounded and belongs to the class $\Lip^0(H,\cL(V,H))$. We have $\sup_{x\in H}\|\varphi(x)\|_{\cL(V,H)}\leq |\psi|_{\Lip^0(H,V)}$ and $|\varphi|_{\Lip^0(H,\cL(V,H))}\leq|\psi|_{\Lip^1(H,V)}$.
An application of Lemma~\ref{lem:op_norm} with $V_1=V$, $V_2=H$ thus yields the assertion. 
\end{proof}

\begin{lemma}\label{lem:op_norm2}
Let $p',q'\in(1,2]$, $F\in L^2(\Omega;H)$ and $S\in\LB(H)$. It holds that 
$\|SF\|_{\bM^{-1,p',q'}(H)}\leq \|S\|_{\LB(H)}\|F\|_{\bM^{-1,p',q'}(H)}$.
\end{lemma}

\begin{proof}
Let $p,q\in[2,\infty)$ satisfy $\tfrac1p+\tfrac1{p'} = \tfrac1q + \tfrac1{q'} =1$. For notational convenience let $B=\bM^{1,p,q}(H)$ and hence $B^*=\bM^{-1,p',q'}(H)$. Assuming without loss of generality that $\|S\|_{\cL(H)}>0$, it holds that
\begin{align*}
  \|SF\|_{B^*}
&=
  \sup_{\|Z\|_{B}=1}
  \langle SF,Z \rangle_{L^2(\Omega;H)}
  =
  \|S^*\|_{\LB(H)}
  \sup_{\|Z\|_{B}=1}
    \Big\langle F,\frac{S^*Z}{\|S^*\|_{\LB(H)}} \Big\rangle_{L^2(\Omega;H)}\\
&\leq    
  \|S^*\|_{\LB(H)}
  \sup_{\|Z\|_{B}=1}
    \langle F,Z \rangle_{L^2(\Omega;H)}
  =
  \|S\|_{\LB(H)}
  \|F\|_{B^*}\qedhere
\end{align*}  

\end{proof}

\section{Weak approximation for a class of semilinear SPDE}
\label{sec:weak}

\subsection{The main result and an application}
\label{subsec:mainresult}

Here we describe the numerical space-time discretization scheme for Eq.~\eqref{eq:SPDE_additive} and formulate our main result on weak convergence in Theorem~\ref{thm:weak2}. For the sake of comparability, we also state a corresponding strong convergence result in Proposition~\ref{prop:strong_error_semilin}. An application of Theorem~\ref{thm:weak2} to covariance convergence in presented in Corollary~\ref{cor:cov}, see \cite{LangKirchnerLarsson2017} for related results.

\begin{assumption}[Discretization]\label{ass:discretization}
For the spatial discretization we use a family
$(V_h)_{h\in(0,1)}$  of finite dimensional subspaces of $H$ and linear operators $A_h\colon V_h\to V_h$ that serve as discretizations of $A$. By $P_h\colon H\to V_h$ we denote the orthogonal projectors w.r.t.\ the inner product in $H$. 
For the discretization in time we use a linearly implicit Euler scheme with uniform grid  $t_m=km$, $m\in\{0,\ldots,M\}$, where  $k\in(0,1)$ is the stepsize and $M=M_k\in\bN$ is determined by $t_M\leq T< t_M+k$. 
The operators 
$S_{h,k}:=(\id_{V_h}+kA_h)^{-1}P_h$
thus serve as discretizations of $S(k)$,
and 
$E_{h,k}^m:=S_{h,k}^m-S(t_m)$
are the corresponding error operators. We assume that there are constants 
$D_\rho,\,D_{\rho,\sigma}\in [0,\infty)$ (independent of $h$, $k$, $m$) such that, 
\begin{align}
&\|A_h^{\frac{\rho}2}S_{h,k}^m \|_{\LB(H)} + \|S_{h,k}^m A^{\frac{\min(\rho,1)}2} \|_{\LB(H)}
\leq
  D_\rho\,t_m^{-\frac{\rho}2},
  \quad
  \rho\geq0,\label{eq:estAh1}\\
&\| E_{h,k}^m A^{\frac{\rho}2}\|_{\LB(H)}
\leq
  D_{\rho,\sigma} \,t_m^{-\frac{\rho+\sigma}2}
  \big(
    h^\sigma + k^{\frac{\sigma}2}
  \big),
  \quad
  \sigma\in[0,2],\ \rho\in [-\sigma,\min(1,2-\sigma)],\label{eq:estAh2}
 \end{align}
for all $h,k\in(0,1)$ and $m\in\{1,\ldots,M\}$.
\end{assumption}

\begin{example}\label{ex:Vh}
In the situation of Example~\ref{ex:A}, the spaces $V_h$ can be chosen as standard finite element spaces consisting of continuous, piecewise linear functions w.r.t.\ regular triangulations of $\cO$, with maximal mesh size bounded by $h$. See, e.g., \cite[Section 5]{AnderssonKruseLarsson} for a proof of the estimates \eqref{eq:estAh1}, \eqref{eq:estAh2} in this case. 
\end{example}

For $h,k\in(0,1)$ and $M=M_k\in\N$  the approximation $(X^m_{h,k})_{m\in\{0,\ldots,M\}}$ of the mild solution $(X(t))_{t\in[0,T]}$ to Eq.~\eqref{eq:SPDE_additive}
is defined recursively by $X^0_{h,k}=P_hX_0$ and
\begin{align}\label{eq:mild_sol_discr_semilin}
X^m_{h,k}
= 
S^m_{h,k}X_0+k\sum_{j=0}^{m-1}S^{m-j}_{h,k}F(X^j_{h,k}) 
+\sum_{j=0}^{m-1}S^{m-j}_{h,k}(L(t_{j+1})-L(t_j)),
\end{align}
$m\in\{1,\dots,M\}$.
By $(\tilde X_{h,k}(t))_{t\in[0,T]}$ we denote the piecewise constant interpolation of $(X_{h,k}^m)_{m\in\{0,\ldots,M\}}$ which is defined as 
\begin{align}\label{eq:mild_sol_discr_interpol_semilin}
\tilde X_{h,k}(t)=\sum_{m=0}^{M-1}\mathbf 1_{[t_{m},t_{m+1})}(t)X_{h,k}^m+\mathbf 1_{[t_M,T]}(t)X^M_{h,k}.
\end{align}

The following strong convergence result can be proven analogously to the Gaussian case, cf.~\cite[Theorem 4.2]{AnderssonKovacsLarsson}.

\begin{prop}[Strong convergence]\label{prop:strong_error_semilin}
Let Assumption~\ref{as:SPDE} hold, let $(X(t))_{t\in[0,T]}$ be the mild solution to Eq.~\eqref{eq:SPDE_additive} and $(\tilde X_{h,k}(t))_{t\in[0,T]}$ be its discretization given by \eqref{eq:mild_sol_discr_semilin}, \eqref{eq:mild_sol_discr_interpol_semilin}.
Then, for every $\gamma\in[0,\beta)$ there exists a constant $C\in [0,\infty)$, 
which does not depend on $h,k$, such that
\begin{align*}
  \sup_{t\in[0,T]}
  \|X(t)-\tilde X_{h,k}(t)\|_{L^2(\Omega;H)}
  \leq
  C(h^\gamma+k^{\frac\gamma2}),
  \quad
  h,k\in(0,1).
\end{align*}
\end{prop}

For the weak convergence we consider path dependent functionals as specified by the next assumption. In the related work \cite{AnderssonKovacsLarsson} functionals of the form
$
  f(x)
  =
  \prod_{i=1}^n
    \varphi_i\big(\int_{[0,T]}x(t)\,\mu_i(\diffin t)\big)
$,
with $\varphi_1,\dots,\varphi_n$ being twice differentiable with polynomially growing derivatives of some fixed but arbitrary degree, and 
$\mu_1,\dots,\mu_n$ being finite Borel measures on $[0,T]$, were considered for equations with Gaussian noise. Here we generalize by removing the product structure, but we only allow for quadratically growing test functions. The reason for the latter restrition is that the solution to our equation has in general only finite moments up to order two while solutions to equations with Gaussian noise have all moments finite.

\begin{assumption}[Test function $f$]\label{as:Phi}
Let $n\in\bN$ and $\varphi\colon\bigoplus_{i=1}^n\!H\to\bR$ be Fréchet differentiable with globally Lipschitz continuous derivative mapping $\varphi'\colon \bigoplus_{i=1}^n\!H\to\cL\big(\bigoplus_{i=1}^n\!H,\bR\big)$.
Let $\mu_1,\ldots,\mu_n$ be finite Borel-measures on $[0,T]$. The functional $f\colon L^1([0,T],\sum_{i=1}^n\mu_i;H)\to\bR$ is given by
\begin{align*}
f(x):=\varphi\Big(\int_{[0,T]}x(t)\,\mu_1(\dl t),\ldots,\int_{[0,T]}x(t)\,\mu_n(\dl t)\Big).
\end{align*}
\end{assumption}

Observe that $X,\,\tilde X_{h,k}\in L^2(\Omega;L^1([0,T],\sum_{i=1}^n\mu_i;H))$ due to \eqref{eq:XL2} and, e.g., the estimate \eqref{eq:lemstab1} below. In particular, the random variables $f(X)$, $f(\tilde X_{h,k})$ are defined and integrable.

We next state our main result on weak convergence. The proof is postponed to Subsections~\ref{subsec:weak1}--\ref{subsec:weak4}. Note that the obtained weak rate of convergence is twice the strong rate from Propostion~\ref{prop:strong_error_semilin}.

\begin{theorem}[Weak convergence]\label{thm:weak2}
Let Assumption~\ref{as:SPDE}, \ref{ass:discretization} and \ref{as:Phi} hold. Let $X=(X(t))_{t\in[0,T]}$ be the mild solution to Eq.~\eqref{eq:SPDE_additive} and $\tilde X_{h,k}=(\tilde X_{h,k}(t))_{t\in[0,T]}$ be its discretization given by \eqref{eq:mild_sol_discr_semilin}, \eqref{eq:mild_sol_discr_interpol_semilin}.
Then, for every $\gamma\in[0,\beta)$ there exists a constant $C\in [0,\infty)$, 
which does not depend on $h,k$, such that the weak error estimate
\eqref{eq:weak_conv} holds. 
\end{theorem}

\begin{corollary}[Covariance convergence]\label{cor:cov}
Consider the setting of Theorem~\ref{thm:weak2}. \linebreak For all $\gamma\in[0,\beta)$, $t_1,t_2\in(0,T]$ and $\phi_1,\phi_2\in H$ there exists a constant $C\in[0,\infty)$, which does not depend on $h,k$, such that
\begin{align*}
\big|
    \mathrm{Cov}\big(\big\langle X(t_1),\phi_1\big\rangle,\big\langle X(t_2),\phi_2\big\rangle\big)
   -
    \mathrm{Cov}\big(\big\langle \tilde X_{h,k}(t_1),\phi_1&\big\rangle,\big\langle \tilde X_{h,k}(t_2),\phi_2\big\rangle\big)
  \big|\\
  &\leq
  C\,(h^{2\gamma}+k^\gamma),\quad h,k\in(0,1).
\end{align*}
\end{corollary}

\begin{proof}[Proof of Corollary~\ref{cor:cov}]
For random variables $Y_1,Y_2,Z_1,Z_2\in L^2(\Omega;H)$ and vectors $\phi_1,\phi_2\in H$ it holds that
\begin{equation}\label{eq:cov_calc}
\begin{split}
&\mathrm{Cov}\big(\langle Y_1,\phi_1 \rangle,\langle Y_2,\phi_2 \rangle\big)
  -
  \mathrm{Cov}\big(\langle Z_1,\phi_1 \rangle,\langle Z_2,\phi_2 \rangle\big)\\
&
  =
  \E\big[
    \langle Y_1,\phi_1 \rangle\langle Y_2,\phi_2 \rangle
    -
    \langle Z_1,\phi_1 \rangle\langle Z_2,\phi_2 \rangle
  \big] 
   - 
    \E\big[\langle Y_1,\phi_1 \rangle
    -
   \langle Z_1,\phi_1\rangle\big]
    \,\E\langle Y_2,\phi_2\rangle\\
&\quad
  -
  \E\langle Z_1,\phi_1\rangle\,
  \E\big[\langle Y_2,\phi_2\rangle-\langle Z_2,\phi_2\rangle\big]
\end{split}
\end{equation}
We consider the Borel measure $\mu:=\delta_{t_1}+\delta_{t_2}$ on $[0,T]$ as well as the functionals $f_i\colon L^2([0,T],\mu;H)\to\bR$, $i\in\{1,2,3\}$, given by
$f_1(x):=\langle x(t_1),\phi_1\rangle\langle x(t_2),\phi_2\rangle$, $f_2(x):=\langle x(t_1),\phi_1\rangle$, $f_3(x):=\langle x(t_2),\phi_2\rangle$.
These functionals satisfy Assumption~\ref{as:Phi}. From \eqref{eq:cov_calc} with $Y_1 = X(t_1)$, $Y_2 = X(t_2)$, $Z_1=\tilde X_{h,k}(t_1)$ and $Z_2=\tilde X_{h,k}(t_2)$ we obtain
\begin{align*}
&\big|
    \mathrm{Cov}\big(\langle X(t_1),\phi_1\rangle,\langle X(t_2),\phi_2\rangle\big)
    -
    \mathrm{Cov}\big(\langle \tilde X_{h,k}(t_1),\phi_1\rangle,\langle \tilde X_{h,k}(t_2),\phi_2\rangle\big)
  \big|\\
&
  \leq
  \big|
    \E
    \big[
      f_1(X)
      -
      f_1(\tilde X_{h,k})
    \big]
  \big|
  +
  \|\phi_2\|
  \sup_{t\in[0,T]}\|X(t)\|_{L^2(\Omega;H)}
  \big|
    \E
    \big[
      f_2(X)
      -
      f_2(\tilde X_{h,k})
    \big]
  \big|\\
&\quad
  +
  \|\phi_1\|
  \sup_{h,k\in(0,1)}
  \|\tilde X_{h,k}(t_1)\|_{L^2(\Omega;H)}
  \big|
    \E
    \big[
      f_3(X)
      -
      f_3(\tilde X_{h,k})
    \big]
  \big|
\end{align*}
Three applications of Theorem~\ref{thm:weak2} together with \eqref{eq:XL2} and the estimate~\eqref{eq:lemstab1} below complete the proof. 
\end{proof}

\subsection{A regularity result for the discrete solution}
\label{subsec:weak2}

Here we prove an analogue of Proposition~\ref{thm:reg2} for the discrete solution. It has Gaussian counterparts in \cite[Proposition 4.3]{AnderssonKovacsLarsson} and \cite[Proposition 3.17]{AnderssonKruseLarsson}.

\begin{prop}\label{lem:stab}
Let Assumption~\ref{as:SPDE} and \ref{ass:discretization} hold. 
Depending on the value of $\beta\in(0,1]$, we assume either that $q\in(1,\tfrac2{1-\beta})$ if $\beta\in(0,1)$ or $q=\infty$ if $\beta=1$.
Then,
\begin{align}\label{eq:DXhk_uniform}
  \sup_{h,k\in(0,1)}\sup_{m\in\{0,\dots,M_k\}}
  \|X^m_{h,k}\|_{\bM^{1,2,q}(H)}
  <\infty.
\end{align}
\end{prop}

\begin{proof}
By a classical Gronwall argument based on Lemma~\ref{lem:discrete_gronwall}, it holds that
\begin{align}\label{eq:lemstab1}
  \sup_{h,k\in(0,1)}\sup_{m\in\{0,\dots,M_k\}}
  \|X^m_{h,k}\|_{L^2(\Omega;H)}
  <\infty.
\end{align}  
Up to some straightforward modifications, the proof of \eqref{eq:lemstab1} is analogous to that of \cite[Proposition 3.16]{AnderssonKruseLarsson} in the Gaussian case and is therefore omitted.
Next, we rewrite the scheme \eqref{eq:mild_sol_discr_semilin} in the form
\begin{align*}
X^m_{h,k}
= 
S^m_{h,k}X_0+k\sum_{j=0}^{m-1}S^{m-j}_{h,k}F(X^j_{h,k}) 
+\sum_{j=0}^{m-1}\int_0^T\int_U\mathds{1}_{(t_j,t_{j+1}]}(s)\,S^{m-j}_{h,k}x\,\tilde N(\dl s,\dl x),
\end{align*}
$m\in\{1,\ldots,M\}$. Applying the difference operator $D$ on the single terms in this equation and taking into account Lemma~\ref{lem:chain}, Lemma~\ref{lem:DF_is_zero} and Proposition~\ref{prop:comm_space-time2}, we obtain

\begin{equation}\label{eq:DXhk}
\begin{aligned}
  D_{s,x}X^m_{h,k}
  &=
  k\sum_{j=\lceil s\rceil_k}^{m-1}S^{m-j}_{h,k}\big[F(X_{h,k}^j+D_{s,x}X_{h,k}^j)-F(X_{h,k}^j)\big]\\
  &\quad+
  \sum_{j=0}^{m-1}\mathds 1_{(t_j,t_{j+1}]}(s)\,S^{m-j}_{h,k}x 
\end{aligned}
\end{equation}
holding $\bP\otimes\dl s\otimes\nu(\dl x)$-almost everywhere on $\Omega\times[0,T]\times U$.
Here we denote for $s\in[0,T]$ by $\lceil s\rceil_k$ is the smallest number $i\in \N$ such that $ik\geq s$. According to Lemma~\ref{lem:DF_is_zero}, the identity $D_{s,x}X_{h,k}^j=0$ holds $\bP\otimes\dl s\otimes\nu(\dl x)$-almost everywhere on $\Omega\times(t_j,T]\times U$.
Taking norms in \eqref{eq:DXhk} yields
\begin{equation}\label{eq:lemstab2}
\begin{aligned}
&\|
    D_{s,x}X^m_{h,k}
  \|_{L^\infty(\Omega;L^q([0,T];L^2(U;H)))}\\
&\leq
  k\sum_{j=0}^{m-1}
   \Big\| S^{m-j}_{h,k}\big[F(X_{h,k}^j+DX_{h,k}^j)-F(X_{h,k}^j)\big]
  \Big\|_{L^\infty(\Omega;L^q([0,T];L^2(U;H)))}\\
&\quad+
  \Big\|
    (s,x)\mapsto \sum_{j=0}^{m-1} \mathds 1_{(t_j,t_{j+1}]}(s)\,S^{m-j}_{h,k}x
  \Big\|_{L^q([0,T];L^2(U;H))}.
\end{aligned}
\end{equation}
Using the estimate \eqref{eq:estAh1} and the Lipschitz assumption on $F$, we obtain
\begin{equation}\label{eq:lemstab3}
\begin{aligned}
& \Big\| S^{m-j}_{h,k}\big[F(X_{h,k}^j+DX_{h,k}^j)-F(X_{h,k}^j)\big]
  \Big\|_{L^\infty(\Omega;L^q([0,T];L^2(U;H)))}\\
&\quad
  \leq
  D_{1-\beta}\,t_{m-j}^{\frac{\beta-1}2}|F|_{\Lip(H,\dot H^{\beta-1})}
    \|DX_{h,k}^j\|_{L^\infty(\Omega;L^q([0,T];L^2(U,H)))}.
\end{aligned}
\end{equation}
Concerning the second term in \eqref{eq:lemstab2} we apply the estimate \eqref{eq:estAh1} together with the identity $U=\dot H^{\beta-1}$ and observe that
\begin{equation}\label{eq:lemstab4}
\begin{aligned}
&\Big\|
    (s,x)\mapsto \sum_{j=0}^{m-1} \mathds 1_{(t_j,t_{j+1}]}(s)\,S^{m-j}_{h,k}x
  \Big\|_{L^q([0,T];L^2(U;H))}\\
&\quad
  =
  \Big(
    \int_0^{T}
      \Big(
        \int_U
          \Big\|  
            \sum_{j=0}^{m-1} \mathds 1_{(t_j,t_{j+1}]}(s)\,S^{m-j}_{h,k}x
          \Big\|^2
        \nu(\diffin x)
      \Big)^\frac{q}2
    \diffout s
  \Big)^{\frac1q}\\
&\quad
  \leq
  D_{1-\beta}\,
  |\nu|_2
  \Big(
    k\sum_{j=0}^{m-1}
      t_{m-j}^{\frac{q(\beta-1)}2}
  \Big)^{\frac1q}
  \leq
  D_{1-\beta}\,
  |\nu|_2
   \Big(
    \int_0^T
      (T-r)^{\frac{q(\beta-1)}2}
    \dl r
  \Big)^{\frac1q}
  <\infty.
\end{aligned}
\end{equation}
The penultimate inequality follows by approximating the sum by a Riemann integral and observing that the singularity is integrable. 
From \eqref{eq:lemstab2}, \eqref{eq:lemstab3} and \eqref{eq:lemstab4} we conclude that for all $m\in\{1,\dots M_k\}$, uniformly in $h,k\in(0,1)$,
\begin{align*}\label{eq:lemstab5}
  \|DX_{h,k}^m\|_{L^\infty(\Omega;L^q([0,T];L^2(U,H)))}
  \lesssim
  1 
  +
  k \sum_{j=0}^{m-1}t_{m-j}^{\frac{\beta-1}2}
    \|DX_{h,k}^j\|_{L^\infty(\Omega;L^q([0,T];L^2(U,H)))}.
\end{align*}
By induction we obtain that $DX_{h,k}^m\in L^{\infty}(\Omega;L^q([0,T];L^2(U;H)))$ for all $m$, so that \eqref{eq:lemstab1} and 
an application of the discrete Gronwall Lemma~\ref{lem:discrete_gronwall} yield the uniform bound \eqref{eq:DXhk_uniform}. 
\end{proof}

\subsection{Convergence in negative order spaces}
\label{subsec:weak3}

The following crucial result has Gaussian counterparts in \cite[Lemma 4.6]{AnderssonKovacsLarsson} and \cite[Lemma 4.6]{AnderssonKruseLarsson}.

\begin{lemma}\label{lem:dual_conv}
Let Assumption~\ref{as:SPDE} and \ref{ass:discretization} hold, let $(X(t))_{t\in[0,T]}$ be the mild solution to Eq.~\eqref{eq:SPDE_additive} and $(\tilde X_{h,k}(t))_{t\in[0,T]}$ be its discretization given by \eqref{eq:mild_sol_discr_semilin}, \eqref{eq:mild_sol_discr_interpol_semilin}. Then, for every $\gamma\in[0,\beta)$ and $q'=\tfrac2{1+\gamma}$ there exists a constant $C\in [0,\infty)$, 
which does not depend on $h,k$, such that
\begin{align*} 
  \sup_{t\in[0,T]}\|\tilde X_{h,k}(t)-X(t)\|_{\bM^{-1,2,q'}(H)}
  \leq
  C\,
  \big(
    h^{2\gamma}
    +
    k^\gamma
  \big),
  \quad
  h,k\in(0,1).
\end{align*}
\end{lemma}

\begin{proof}
For notational convenience we introduce the piecewise continuous error mapping $\tilde E_{h,k}\colon [0,T)\to\LB(H)$ given by $\tilde E_{h,k}(t):=S_{h,k}^m-S(t)$ for $t\in[t_{m-1},t_m)$, so that
\begin{align*}
  X_{h,k}^m - X(t_m)
&=
  E_{h,k}^mX_0
  +
  \int_0^{t_m}
    \tilde E_{h,k}(t_m-s)
    F(\tilde X_{h,k}(s))
  \diffout s\\
&\quad
  +
  \int_0^{t_m}
    S(t_m-s)
    \big(
      F(\tilde X_{h,k}(s))-F(X(s))
    \big)
  \diffout s\\
&\quad
  +
  \int_0^{t_m}\int_{\dot H^{\beta-1}}
    \tilde E_{h,k}(t_m-s)x\,\tilde N(\dl s,\dl x).
\end{align*}
Taking norms and using the continuous embedding 
$L^2(\Omega;H)\subset \bM^{-1,2,q'}(H)$ as well as Minkowski's integral inequality yields
\begin{equation}\label{eq:proofnegnorm1}
\begin{aligned}
&\big\|
    X_{h,k}^m - X(t_m)
  \big\|_{\bM^{-1,2,q'}(H)}\\
&\quad
  \leq
  \|
    E_{h,k}^mX_0
  \|
  +
  \int_0^{t_m}
    \big\|
      \tilde E_{h,k}(t_m-s)
      F(\tilde X_{h,k}(s))
    \big\|_{L^2(\Omega;H)}
  \diffout s\\
&\qquad
  +
  \int_0^{t_m}
    \big\|
      S(t_m-s)
      \big(
        F(\tilde X_{h,k}(s))-F(X(s))
      \big)
    \big\|_{\bM^{-1,2,q'}(H)}
  \diffout s\\
&\qquad
  +
  \Big\|
    \int_0^{t_m}\int_{\dot H^{\beta-1}}
    \tilde E_{h,k}(t_m-s)x\,\tilde N(\dl s,\dl x)
  \Big\|_{\bM^{-1,2,q'}(H)}.
\end{aligned}
\end{equation}

We estimate the terms on the right hand side separately. To this end, note that the error estimate \eqref{eq:estAh2} extends to the piecewise continuous error mapping $\tilde E_{h,k}$. Indeed, as a consequence of the identity $\tilde E_{h,k}(t)=E_{h,k}^m+(S(t_m)-S(t))$, $t\in[t_{m-1},t_m)$, and the estimates \eqref{eq:smoothing}, \eqref{eq:continuity}, \eqref{eq:estAh2}, we have
\begin{align}\label{eq:estAh3}
\big\|\tilde E_{h,k}(t)A^{\frac\rho2}\big\|_{\LB(H)}\leq (D_{\rho,\sigma}+C_\sigma C_{\sigma+\rho})\,t^{-\frac{\rho+\sigma}2}\big(h^\sigma+k^{\frac\sigma2}\big),
\end{align}
holding for $\sigma\in[0,2]$, $\rho\in [-\sigma,\min(1,2-\sigma)]$ and $h,k\in(0,1)$, $t\in(0,T]$.

Concerning the first two terms on the right hand side of \eqref{eq:proofnegnorm1} we observe that \eqref{eq:continuity}, 
\eqref{eq:estAh3}, and the linear growth of $F$ yield
\begin{equation}\label{eq:proofnegnorm2}
\begin{aligned}
&\|
    E_{h,k}^mX_0
  \|
  +
  \int_0^{t_m}
    \big\|
      \tilde E_{h,k}(t_m-s)
      F(\tilde X_{h,k}(s))
    \big\|_{L^2(\Omega;H)}
  \diffout s\\
&\leq
  D_{-2\gamma,2\gamma}\|X_0\|_{\dot H^{2\gamma}}
  \big(
    h^{2\gamma} + k^\gamma
  \big)
  +  
    (D_{1-\beta,2\gamma}+C_{2\gamma}C_{2\gamma+1-\beta})
    \frac{T^{\frac{1+\beta}2-\gamma}}
  {
    (1+\beta)/2-\gamma
  }\\
&\quad
  \cdot\|F\|_{\Lip^0(H,\dot H^{\beta-1})}
   \big(1+\sup_{t\in[0,T]}\|X(t)\|_{L^2(\Omega;H)}\big)
  \big(
    h^{2\gamma} + k^\gamma
  \big).
\end{aligned}
\end{equation}

Next, we use Lemma~\ref{lem:op_norm2}, \eqref{eq:smoothing} and Proposition~\ref{prop:Lip2} to estimate the third term on the right hand side of \eqref{eq:proofnegnorm1} from above by
\begin{equation}\label{eq:proofnegnorm3}
\begin{aligned}
  & 
  \int_0^{t_m}
    \big\|
      S(t_m-s)
      A^{\frac\delta2}
    \big\|_{\LB(H)}
    \big\|
      A^{-\frac\delta2}
      \big(
        F(\tilde X_{h,k}(s))
        -
        F(X(s))
      \big)
    \big\|_{\bM^{-1,2,q'}(H)}
  \diffout s\\
&\quad
  \leq
  C_{\delta}
  K
  \sum_{i=0}^{m-1}
  \int_{t_i}^{t_{i+1}}
    (t_m-s)^{-\frac\delta2}
    \big\|
      X_{h,k}^i
      -
      X(t_i)
    \big\|_{\bM^{-1,2,q'}(H)}
  \diffout s\\
&\qquad
  +
  C_{\delta}
  K
  \sum_{i=0}^{m-1}
  \int_{t_i}^{t_{i+1}}
    (t_m-s)^{-\frac\delta2}
    \big\|
      X(t_i)
      -
      X(s)
    \big\|_{\bM^{-1,2,q'}(H)}
  \diffout s,
\end{aligned}
\end{equation}
where $K$ is, by Proposition~\ref{thm:reg2}, Proposition~\ref{prop:Lip2} and Lemma~\ref{lem:stab}, the finite constant
\begin{align*}
  K
&=
  4\Big(
    |F|_{\Lip^0(H,\dot H^{-\delta})} 
    + |F|_{\Lip^1(H,\dot H^{-\delta})}\sup_{s\in[0,T]}|X(s)|_{\bM^{1,\infty,q}(H)}\\
& \qquad    
    + |F|_{\Lip^1(H,\dot H^{-\delta})} \sup_{h,k\in(0,1)}\sup_{m\in\{0,\ldots,M_k\}}
    |X_{h,k}^m|_{\bM^{1,\infty,q}(H)}
  \Big)<\infty.
\end{align*}
The terms on the right hand side of \eqref{eq:proofnegnorm3} can be estimated as follows: We have
\begin{align*}
&\sum_{i=0}^{m-1}
  \int_{t_i}^{t_{i+1}}
    (t_m-s)^{-\frac\delta2}
    \big\|
      X_{h,k}^i
      -
      X(t_i)
    \big\|_{\bM^{-1,2,q'}(H)}
  \diffout s\\
&
  \leq
  k
  \sum_{i=0}^{m-2}
    t_{m-i-1}^{-\frac\delta2}
    \big\|
      X_{h,k}^i
      -
      X(t_i)
    \big\|_{\bM^{-1,2,q'}(H)}
    +
    \frac{k^{1-\frac\delta2}}{1-\frac{\delta}2}
    \big\|
      X_{h,k}^{m-1}
      -
      X(t_{m-1})
    \big\|_{\bM^{-1,2,q'}(H)}.    
\end{align*}
Since for all $m\in\{2,3,\ldots\}$ it holds that
$
  \max_{i\in\{0,1,\dots, m-2\}}\big(
    t_{m-i-1}^{-\frac\delta2}  \cdot
    t_{m-i}^{\frac\delta2}  \big)
  =$\linebreak$
  \max_{i\in\{0,1,\dots, m-2\}}
    \big((m-i-1)^{-\frac\delta2}\cdot
    (m-i)^{\frac\delta2}\big)
  =
  2^{\frac\delta2},
$
we obtain for all $m\in\N=\{1,2,\ldots\}$
\begin{equation}\label{eq:proofnegnorm3a}
\begin{aligned}
&\sum_{i=0}^{m-1}
  \int_{t_i}^{t_{i+1}}
    (t_m-s)^{-\frac\delta2}
    \big\|
      X_{h,k}^i
      -
      X(t_i)
    \big\|_{\bM^{-1,2,q'}(H)}
  \diffout s\\
&
  \leq
  2^{\frac\delta2}
  k
  \sum_{i=0}^{m-2}
    t_{m-i}^{-\frac\delta2}
    \big\|
      X_{h,k}^i
      -
      X(t_i)
    \big\|_{\bM^{-1,2,q'}(H)}
    +
    \frac{k^{1-\frac\delta2}}{1-\frac{\delta}2}
    \big\|
      X_{h,k}^{m-1}
      -
      X(t_{m-1})
    \big\|_{\bM^{-1,2,q'}(H)}\\
&
  \leq
\frac{  
  2^{\frac\delta2}
  k
  }
  {1-\frac{\delta}2}
  \sum_{i=0}^{m-1}
    t_{m-i}^{-\frac\delta2}
    \big\|
      X_{h,k}^i
      -
      X(t_i)
    \big\|_{\bM^{-1,2,q'}(H)}.    
\end{aligned}
\end{equation}
Moreover, by the H\"older continuity of Proposition~\ref{thm:reg3} it holds
\begin{equation}\label{eq:proofnegnorm3a}
\begin{aligned}
  \sum_{i=0}^{m-1}
  \int_{t_i}^{t_{i+1}}
    (t_m-s)^{-\frac\delta2}
    \big\|
      X(t_i)
      -
      X(s)
    \big\|_{\bM^{-1,2,q'}(H)}
  \diffout s
  \leq
  C
  k^{\gamma}
  \int_{0}^{t_m}
    (t_m-s)^{-\frac\delta2}
  \diffout s
  \lesssim k^{\gamma}.
\end{aligned}
\end{equation}

Concerning the fourth term on the right hand side of \eqref{eq:proofnegnorm1}, note that the negative norm inequality in Proposition~\ref{prop:negnorm2} yields
\begin{equation}\label{eq:proofnegnorm4}
\begin{aligned}
&\Big\|
    \int_0^{t_m}\int_{\dot H^{\beta-1}}
      \tilde E_{h,k}(t_m-s)x\,
    \tilde N(\dl s,\dl x)
  \Big\|_{\bM^{-1,2,q'}(H)}\\
&\qquad
  \leq
  \Big[
    \int_0^{t_m}\Big(\int_{\dot H^{\beta-1}}
      \|\tilde E_{h,k}(t_m-s)x\|^{2}\,
    \nu(\dl x)\Big)^{\frac{q'}2}\diffout s
  \Big]^{\frac1{q'}}\\
&\qquad
  \leq
  D_{1-\beta,2\gamma}
  \,|\nu|_2\,
  \big(
    h^{2\gamma} + k^\gamma
  \big)
  \Big(
    \int_0^{t_m}
      (t_m-s)^{\frac2{1+\gamma}\frac{\beta-1-2\gamma}2}
    \diffout s
  \Big)^{\frac{1+\gamma}2}
  \lesssim
  h^{2\gamma} + k^\gamma,  
\end{aligned}
\end{equation}
where the last integral is finite due to \eqref{eq:par_calc}.

Combining the estimates \eqref{eq:proofnegnorm1}--\eqref{eq:proofnegnorm4} yields
\begin{align*}
&\big\|
    X_{h,k}^m - X(t_m)
  \big\|_{\bM^{-1,2,q'}(H)}
  \lesssim
  h^{2\gamma} + k^\gamma
  +
  k
  \sum_{i=0}^{m-1}
    t_{m-i}^{-\frac\delta2}
    \big\|
      X_{h,k}^i
      -
      X(t_i)
    \big\|_{\bM^{-1,2,q'}(H)}.
\end{align*}
The discrete Gronwall Lemma~\ref{lem:discrete_gronwall} thus implies that there exists $C'\in[
0,\infty)$, which does not depend on $h,k$, such that
$
  \sup_{m\in\{1,\dots,M_k\}}
  \|X^m_{h,k}-X(t_m)\|_{\bM^{-1,2,q'}(H)}
  \leq
  C'
  \big(
    h^{2\gamma}
    +
    k^\gamma
  \big)$
   for all $
  h,k\in(0,1).
$
This and Proposition~\ref{thm:reg3} imply the claimed assertion.
\end{proof}

\subsection{Proof of the main result}
\label{subsec:weak4}
We are finally prepared to prove the weak convergence result in Theorem~\ref{thm:weak2}. Recall from Subsection~\ref{subsec:mainresult} that the processes $X=(X(t))_{t\in[0,T]}$ and $\tilde X_{h,k}=(\tilde X_{h,k}(t))_{t\in[0,T]}$ belong to $L^2(\Omega;L^1([0,T],\sum_{i=1}^n\mu_i;H))$.

To simplify notation, we introduce the
$(\bigoplus_{i=1}^n\!H)$-valued random variables $Y=(Y^{(1)},\ldots,Y^{(n)})$, $\tilde Y_{h,k}=(\tilde Y_{h,k}^{(1)},\ldots,\tilde Y_{h,k}^{(n)})$ and $\Phi_{h,k}=(\Phi^{(1)}_{h,k},\ldots,\Phi^{(n)}_{h,k})$ defined by
\begin{equation}\label{eq:defPhi}
\begin{aligned}
Y^{(i)}&:=\int_{[0,T]}X(t)\,\mu_i(\dl t),\quad \tilde Y^{(i)}_{h,k}:=\int_{[0,T]}\tilde X_{h,k}(t)\,\mu_i(\dl t),\\
\Phi^{(i)}_{h,k}&:=\int_0^1\partial_i\varphi\big((1-\theta)Y+\theta\tilde Y_{h,k}\big)\,\dl\theta.
\end{aligned}
\end{equation}
Here we denote for $x=(x^{(1)},\ldots,x^{(n)})\in\bigoplus_{j=1}^n\!H$ by $\partial_i\varphi(x)=\frac{\partial}{\partial x^{(i)}}\varphi(x)$ the Fréchet derivative of $\varphi$ w.r.t.\ the $i$-th coordinate of $x$, considered as an element of $H$ via the Riesz isomorphism $\mathcal L(H,\bR)\equiv H$. Moreover, we set set $q:=\tfrac2{1-\gamma}$ and $q':=\tfrac2{1+\gamma}$.

Using the notation above, the fundamental theorem of calculus, and duality in the Gelfand triple $\bM^{1,2,q}(H)\subset L^2(\Omega;H)\subset \bM^{-1,2,q'}(H)$, we represent and estimate the weak error as follows:
\begin{equation}\label{eq:proofWeakError}
\begin{aligned}
&\big|\bE\big[f(\tilde X_{h,k})-f(X)\big]\big|
=\big|\bE\big[\varphi(\tilde Y_{h,k})-\varphi(Y)\big]\big|
=\Big|\bE\sum_{i=1}^n\big\langle\Phi^{(i)}_{h,k},\tilde Y^{(i)}_{h,k}-Y^{(i)}\big\rangle\Big|\\
&\quad=\Big|\sum_{i=1}^n\int_{[0,T]}\bE\big\langle\Phi^{(i)}_{h,k},\tilde X_{h,k}(t)-X(t)\big\rangle\,\mu_i(\dl t)\Big|\\
&\quad\leq \sum_{i=1}^n\mu_i([0,T])\,\big\|\Phi^{(i)}_{h,k}\big\|_{\bM^{1,2,q}(H)}\sup_{t\in[0,T]}\big\|\tilde X_{h,k}(t)-X(t)\big\|_{\bM^{-1,2,q'}(H)}
\end{aligned}
\end{equation}
The assertion of Theorem~\ref{thm:weak2} now follows from \eqref{eq:proofWeakError}  together with Lemma~\ref{lem:dual_conv} and  Lemma~\ref{lem:estPhihk} below.

\begin{lemma}\label{lem:estPhihk}
Let Assumption~\ref{as:SPDE}, \ref{ass:discretization} and \ref{as:Phi} hold. 
Let $(X(t))_{t\in[0,T]}$ be the mild solution to Eq.~\eqref{eq:SPDE_additive}, $(\tilde X_{h,k}(t))_{t\in[0,T]}$ be its discretization given by \eqref{eq:mild_sol_discr_semilin}, \eqref{eq:mild_sol_discr_interpol_semilin}, and let $\Phi^{(i)}_{h,k}$, $i\in\{1,\ldots,n\}$, $h,k\in(0,1)$ be the $H$-valued random variables defined by \eqref{eq:defPhi}.
For all $\gamma\in[0,\beta)$ and $q=\tfrac2{1-\gamma}$ it holds that
\begin{align*}
  \max_{i\in\{1,\ldots,n\}}\sup_{h,k\in(0,1)}
  \big\|\Phi^{(i)}_{h,k}\big\|_{\bM^{1,2,q}(H)}
  <\infty.
\end{align*}
\end{lemma}

\begin{proof}
First note that the linear growth of $\partial_i\varphi\colon \bigoplus_{j=1}^n\!H\to H$, the estimates \eqref{eq:XL2}, \eqref{eq:lemstab1}, and the fact that $\mu_i([0,T])<\infty$ imply  for all $i\in\{1,\ldots,n\}$ that 
$\sup_{h,k\in(0,1)}\big\|\Phi^{(i)}_{h,k}\big\|_{L^2(\Omega;H)}<\infty$. 
It remains to check that
$
\sup_{h,k\in(0,1)}
  \big|\Phi^{(i)}_{h,k}\big|_{\bM^{1,2,q}(H)}
$
is finite.
The chain rule from Lemma~\ref{lem:chain}, applied to the function $h\colon(\bigoplus_{j=1}^n\!H)\oplus (\bigoplus_{j=1}^n\!H)\to H,\,(y,\tilde y)\mapsto \int_0^1\partial_i\varphi\big((1-\theta)y+\theta\tilde y\big)\,\dl\theta$, yields for all $i\in\{1,\ldots,n\}$ 
\begin{align*}
&D_{s,x}\Phi^{(i)}_{h,k}
=
D_{s,x}\int_0^1\partial_i\varphi\big((1-\theta)Y+\theta\tilde Y_{h,k}\big)\,\dl\theta\\
&=
\int_0^1\!\Big[\partial_i\varphi\Big(\!(1-\theta)\big(Y+D_{s,x}Y\big)\!+\theta\big(\tilde Y_{h,k}+D_{s,x}\tilde Y_{h,k}\big)\Big)\!-\partial_i\varphi\big((1-\theta)Y+\theta\tilde Y_{h,k}\big)\Big]\dl\theta
\end{align*}
$\bP\otimes\dl s\otimes\nu(\dl x)$-almost everywhere on $\Omega\times[0,T]\times U$.
This, the global Lipschitz continuity of $\partial_i\varphi\colon\bigoplus_{j=1}^n\!H\to H$, and Proposition~\ref{lem:intDX} imply
\begin{align*}
\big\|D_{s,x}\Phi^{(i)}_{h,k}\big\|
&\leq |\varphi|_{\Lip^1(\bigoplus_{j=1}^n\!H;\bR)}\Big(\big\| D_{s,x}Y^{(i)}\big\|+\big\| D_{s,x}\tilde Y^{(i)}_{h,k}\big\|\Big)\\
&\leq |\varphi|_{\Lip^1(\bigoplus_{j=1}^n\!H;\bR)}\int_{[0,T]}\Big(\big\| D_{s,x}X(t)\big\|+\big\| D_{s,x}\tilde X_{h,k}(t)\big\|\Big)\mu_i(\dl t)
\end{align*}
Iterated integration w.r.t.\ $\nu(\dl x)$, $\dl s$, $\bP$, and three applications of Minkowski's integral inequality lead to
\begin{align*}
\big|\Phi^{(i)}_{h,k}\big|_{\bM^{1,2,q}(H)}
&\leq |\varphi|_{\Lip^1(\bigoplus_{j=1}^n\!H;\bR)}\int_{[0,T]}\!\Big(|X(t)|_{\bM^{1,2,q}(H)}+\big| \tilde X_{h,k}(t)\big|_{\bM^{1,2,q}(H)}\Big)\mu_i(\dl t)\\
\end{align*}
The estimates \eqref{eq:additive_reg}, \eqref{eq:DXhk_uniform} and the assumption that $\mu_i([0,T])<\infty$ thus imply for all $i\in\{1,\ldots,n\}$ the finiteness of $\sup_{h,k\in(0,1)}\big|\Phi^{(i)}_{h,k}\big|_{\bM^{1,2,q}(H)}$.
\end{proof}

\section*{Acknowledgement}

\appendix
Kristin Kirchner, Raphael Kruse, Annika Lang and Stig Larsson are gratefully acknowledged for participating in early discussion regarding this work and \cite{AnderssonLindner2017a}.

\section{Gronwall Lemmata}

In this section we state two versions of Gronwall's lemma. The first one follows from the arguments in the proof of \cite[Lemma 6.3]{elliott1992} together with the standard version of Gronwall's lemma for measurable functions.
The second one is a slight modification of \cite[Lemma~A.4]{kruse2013}, compare also \cite[Lemma 7.1]{elliott1992}.

\begin{lemma}[Generalized Gronwall lemma]\label{lem:gronwall}
Let $T\in(0,\infty)$ and $\phi\colon\{(t,s):0\leq s\leq t\leq T\}\to[0,\infty)$ be a Borel measurable function satisfying
$\int_s^T\phi(r,s)\,\dl r<\infty$ for all $s\in[0,T]$. If
\begin{align*}
\phi(t,s)\leq A\,(t-s)^{-1+\alpha}+B\int_s^t(t-r)^{-1+\beta}\phi(r,s)\,\dl r,\quad 0\leq s\leq t\leq T,
\end{align*}
for some constants $A,B\in[0,\infty)$, $\alpha,\beta\in(0,\infty)$, then there exists a constant $C=C(B,T,\alpha,\beta)\in[0,\infty)$ such that
$
\phi(t,s)\leq C\,A\,(t-s)^{-1+\alpha},\;0\leq s\leq t\leq T.
$
\end{lemma}

\begin{lemma}[Discrete Gronwall lemma]\label{lem:discrete_gronwall}
Let $T\in(0,\infty)$, $k\in(0,1)$ and $M=M_k\in\bN$ be such that $M k\leq T<(M+1)k$, and set $t_m:=mk$, $m\in\{0,\ldots,M\}$. Let $(\phi_i)_{i=0}^{M}$ be a sequence of nonnegative real numbers.
If
\begin{align*}
  \phi_m\leq A + B\,k\sum_{i=0}^{m-1}t_{m-i}^{-1+\beta}\phi_i,
  \quad
  m\in\{0,\dots,M\},
\end{align*}
for some constants $A,B\in[0,\infty)$, $\beta\in(0,1]$, then there exists a constant $C=C(B,T,\beta)\in[0,\infty)$ such that $\phi_m\leq C\,A$, $m\in \{0,\dots,M\}$.
\end{lemma}
\color{black}

\bibliographystyle{plain}
\bibliography{../litLevy}

\end{document}